\documentclass[a4paper,11pt]{article}

\usepackage{float}
\usepackage{breqn}
\usepackage{epsfig}
\usepackage{subcaption}

\numberwithin{equation}{section}

\usepackage{amsmath,amsfonts,amssymb,amsxtra,bm,setspace,xspace,graphicx,lmodern,psfrag,epsfig,color,latexsym,stmaryrd,scalerel,stackengine,amsthm, enumerate,enumitem,cancel}

\newtheorem{thm}{Theorem}[section]

\newtheorem{Lem}{Lemma}[section]

\newtheorem{Rem}{Remark}[section]
\newtheorem{Exm}{Example}[section]

\newcommand{\bea}{\begin{eqnarray}}
	\newcommand{\eea}{\end{eqnarray}}
\newcommand{\beq}{\begin{equation}}
	\newcommand{\eeq}{\end{equation}}
\newcommand{\beas}{\begin{eqnarray*}}
	\newcommand{\eeas}{\end{eqnarray*}}
\usepackage[colorlinks=true]{hyperref}
\hypersetup{urlcolor=blue, citecolor=blue}

\newcommand{\half}{{\frac{1}{2}}}

\newcommand{\xb}{\color{blue}}

\title{Error analysis with exponential decay estimates for a fully discrete approximation of 
a class of strongly damped wave equations}
\author{Krishan Kumar, P. Danumjaya, Anil Kumar  and Amiya K. Pani\\ %\thanks{
Department of Mathematics, Birla Institute of Technology and Science, Pilani\\
K K Birla Goa Campus, Zuarinagar, Sancoale, Goa-403726, India.\\ 
Email: krishan1624@gmail.com,
danu@goa.bits-pilani.ac.in; \\
anilpundir@goa.bits-pilani.ac.in and 
amiyap@goa.bits-pilani.ac.in} 
\date{}
\allowdisplaybreaks
\begin{document}
\maketitle
\begin{abstract}
This paper deals with the asymptotic behavior and FEM error analysis of a class of strongly damped wave equations using a semidiscrete finite element method in spatial directions combined with a finite difference scheme in the time variable. For the continuous problem under weakly and strongly damping parameters $\alpha$ and $\beta,$ respectively, a novel approach usually used for linear parabolic problems is employed to derive an exponential decay property with explicit rates, which depend on model parameters and the principal eigenvalue of the associated linear elliptic operator for the different cases of parameters such as $(i) \;\alpha, \beta >0$, $ (ii)\; \alpha>0, \beta \geq 0$ and $(iii)\;\alpha \geq 0, \beta >0$. Subsequently, for a semi-discrete finite element scheme keeping the temporal variable continuous, optimal error estimates are derived that preserve exponential decay behavior. Some generalizations that include forcing terms and spatially as well as time-varying damping parameters are discussed. Moreover, an abstract discrete problem is discussed, and as a consequence, uniform decay estimates for finite difference as well as spectral approximations to the damped system are briefly indicated. A complete discrete scheme is developed and analyzed after applying a finite difference scheme in time, which again preserves the exponential decay property. The given proofs involve several energies with energy-based techniques to derive the consistency between continuous and discrete decay rates, in which the constants involved do not blow up as $\alpha\to 0$ and $\beta\to 0$. Finally, several numerical experiments are conducted whose results support the theoretical findings, illustrate uniform decay rates, and explore the effects of parameters on stability and accuracy. 
\end{abstract}
%%%%%%%%%%%%%%%%%%%%%%%%%%%%%%%%%%%%%%%%%%%%%%%%%%%%%%%%%%%%%%%%%%%%%%%%%
%%%%%%%%%%%%%%%%%%%%%%%%%%%%%%%%%%%%%%%%%%%%%%%%%%%%%%%%%%%%%%%%%%%%%%%%%	

%%%%%%%%%%%%%%%%%%%%%%%%%%%%%%%%%%%%%%%%%%%%%%%%%%%%%%%%%%%%%%%%%%%%%%%%%
%%%%%%%%%%%%%%%%%%%%%%%%%%%%%%%%%%%%%%%%%%%%%%%%%%%%%%%%%%%%%%%%%%%%%%%%%	
\section{Introduction}
%%%%%%%%%%%%%%%%%%%%%%%%%%%%%%%%%%%%%%%%%%%%%%%%%%%%%%%%%%%%%%%%%%%%%%%%%
%%%%%%%%%%%%%%%%%%%%%%%%%%%%%%%%%%%%%%%%%%%%%%%%%%%%%%%%%%%%%%%%%%%%%%%%
 This paper focuses on finite element error analysis with the asymptotic behavior of a class of strongly damped wave equations 
 %with a mixture of internal and external damping 
 of the following type: Find $u = u(x,t), x \in \Omega, t \geq 0$ such that
\begin{equation}\label{se1.1}
	u'' + \beta \, A u' + \alpha\,u' + A u = 0 \;\; \mbox{in} \;\;  \Omega \times (0, \infty),
\end{equation}
with boundary condition
\begin{equation}\label{se1.2}
	u(x,t) = 0 \;\; \mbox{on} \;\; \partial \Omega \times [0, \infty),
\end{equation}
and initial conditions
\begin{equation}\label{se1.3}
	u(x,0) = u_0  \;\; \mbox{and} \;\; u'(x,0) = u_1 \;\; \mbox{in} \;\; \Omega,
\end{equation}
where $\Omega$ is a bounded domain in $\mathbb{R}^d$ with boundary $\partial \Omega$, $u' := \frac{\partial u}{\partial t}$, $u'' := \frac{\partial^2 u}{\partial t^2}$ and $A := -\Delta$ with $\Delta$ denoting Laplacian. Here, $\beta$ and $\alpha$ are damping parameters with the following properties: $(i) \, \alpha>0, \;\beta>0$, $(ii) \, \alpha > 0, ~ \beta \geq 0$ or $(iii) \, \alpha \geq 0, \beta > 0$. When $\alpha>0$ and $\beta=0$, the equation \eqref{se1.1} is called weakly damped, and for $\alpha=0$ and $\beta>0$, it is called the strongly damped wave equation. Such equations and their nonlinear versions arise in a large number of applications, say, for example, in the study of noise dampening, acoustics, viscoelastic materials \cite{dell2011long, MR3476505}, railway tracks \cite{8392747}, quantum mechanics \cite{ghidaglia1991longtime}, etc., and their well-posedness with asymptotic properties; see \cite{webb1980existence, massatt1983limiting, AD1988,carvalho2008strongly, danumjaya2023asymptotic, PS2005, PZ2006} and references therein.

Earlier, Russell \cite{russell1975decay} established the explicit non-uniform decay rate for the weakly damped wave equation ($\beta = 0$) using a control-theoretic approach. Subsequently, based on the positive principal eigenvalue of the operator $A$, the decay estimate for the weakly damped wave equation $(\beta = 0, \alpha > 0)$ has been discussed in \cite[Proposition 1.2 of Chapter 4]{temam2012infinite}.  Thereafter, Rauch \cite{rauch1976qualitative} has derived the decay estimate based on the minimum and maximum of the damping coefficient $(\beta = 0, \alpha = \alpha(x) > 0)$ and the principal eigenvalue of the operator $A$. For related papers, see \cite{chen1990asymptotic, chen1979control, chen1981control} and references therein.

On numerical solution, in an excellent paper \cite{larsson1991finite}, Larsson {\it et al.}  have derived appropriate regularity results with minimal smoothness assumption on the initial data and have used these regularity results to provide optimal order of convergence for the semi-discrete finite element method combined with fully discrete scheme for \eqref{se1.1}. It is shown that the results are valid uniformly in time.  Later, T\'ebou {\it et al.} \cite{tebou2003uniform} have proposed a semi-discrete finite difference scheme for the weakly damped wave equation (i.e., $\beta = 0$ in \eqref{se1.1}) and have shown that the discrete energy decays exponentially with decay rate depending on the mesh size. But, by adding an artificial numerical viscosity term (more like a strongly damped term with $\beta$ depending on the mesh parameter), they have derived uniform exponential decay of the discrete energy in \cite{munch2007uniform}. Rincon {\it et al.} \cite{rincon2013numerical} in $2013$ have proved the existence, uniqueness, and regularity of the solution of an auxiliary locally damped wave equation that includes an artificial viscosity term. They have analyzed a semidiscrete scheme based on the finite element method and discussed a fully discrete numerical scheme using implicit time discretization.  However, the asymptotic behavior of the discrete scheme remains unresolved. In $2019$, Egger {\it et al.} \cite{egger2019uniform} have studied the weakly damped linear wave equation (i.e., $\beta = 0$ in \eqref{se1.1}) numerically by employing the mixed finite element method to preserve the uniform exponential decay of the first energy norm as well as higher order energy norms. 
%A crucial point of their proof is the energy argument and the use of the Poincar\'e type inequality. 
Recently, Danumjaya {\it et al.} \cite{danumjaya2023asymptotic} have also derived the exponential decay of the first and higher order energy norms of the solution to a weakly damped wave equation (i.e., $\beta = 0$ in \eqref{se1.1}). Based on the conforming finite element method, it is shown that the semi-discrete scheme preserves the uniform exponential decay property.  Moreover,  optimal orders of convergence that preserve the exponential rate have been established for the semi-discrete approximations. One can also refer to \cite{ljung2021generalized, zhao2023low} and references therein.

The starting point of our investigation is to prove exponential decay behavior and to establish optimal error estimates for a class of strongly damped wave equations  \eqref{se1.1}-\eqref{se1.3}. While in engineering applications related to the damped system, it is important to understand the rate of decay, which depends on the damping parameters $\alpha$ and $\beta$ of the damping materials used. This paper provides a range of decay estimates that may help in applications for choosing a material property-related damping parameter. However, the range given in this paper may not be optimal, as shown in the section on numerical experiments. For optimality, the Fourier analysis given in \cite{larsson1991finite} using eigenvalue expansion may be useful, but it may be restricted for many applications such as parameters depending on space and time; therefore, an attempt has been made in this paper to provide energy-based techniques to establish consistency between the continuous and discrete decay rates and explicit rates depend on damping parameters and also on the principal eigenvalue of the associated linear elliptic operator. The rigorous analysis and careful estimates significantly improve our results.  The highlights of our main contributions are as follows:
\begin{itemize}
\item A novel approach, usually employed for analysis of decay behavior of linear parabolic equations, is developed for deriving uniformly exponential decay estimates for the continuous problem under weakly damping parameter $\alpha $ and strongly damping parameter $\beta$ with properties $ (i)\;\alpha, \;\beta >0$, $(ii)\; \alpha >0, \;\beta \geq 0$ and $(iii) \; \alpha\geq 0,\; \beta >0.$  Further, it identifies how the decay rates depend on model parameters and the first eigenvalue of the associated linear elliptic operator. Compared with the paper by Danumjaya {\it et al.} \cite{danumjaya2023asymptotic} on a weakly damped equation, that is, $\beta=0$, the present analysis provides a better decay rate.

\item Then, a semidiscrete finite element scheme is applied keeping the time variable continuous, which preserves the exponential decay property, and it establishes optimal error estimates in $L^{\infty}(L^2(\Omega))$ and $L^{\infty}(H^1(\Omega))$-norms under minimal smoothness of the initial data, that is, $u_0\in H^3\cap H^1_0(\Omega)$ and $u_1\in H^2(\Omega),$ see \cite{danumjaya2023asymptotic, Rauch}.

\item Some generalizations are indicated, which include the forcing term, spatially as well as time-varying damping parameters. Furthermore, an abstract discrete framework is developed, and as a fallout, uniform decay estimates for finite difference as well as spectral approximations of the damped system are indicated.

\item A completely discrete scheme is analyzed after discretizing the time variable by a finite difference method, and an exponential decay estimate, which is uniform with respect to the discretizing parameters, is derived using a careful analysis.

\item  Finally, several numerical experiments are included that support the theoretical findings, illustrate uniform decay rates, and explore parameter effects on stability and accuracy.
\end{itemize}

Overall, the paper aims to enhance the understanding of how finite element approximations to damped wave equations behave asymptotically and establishes that suitable discretizations can preserve the essential decay properties, as seen in the continuous case. It is further noted that all results are valid uniformly as $\alpha\to 0$ with $\beta > 0$ or as $\beta\to 0$ with $\alpha > 0.$

We use the standard notation for Sobolev spaces and norms. In particular, let $L^2(\Omega)$ denote the space of square integrable functions on $\Omega$ with natural inner product $(\cdot, \cdot)$ and induced norm $\|\cdot\|$. For a nonnegative integer $k$, let $H^k(\Omega)$ denote the Hilbert Sobolev space with the norm $\|\cdot\|_k$. Let
\[
H_0^1(\Omega) = \{ \Psi \in H^1(\Omega) : \Psi = 0 \;\;\mbox{on}\; 
\partial \Omega \}. 
\]

$A = -\Delta$ is a linear self-adjoint positive definite operator defined on $L^2(\Omega)$ with dense domain $D(A) = H^2(\Omega) \cap H_0^1(\Omega)$. We set $D(A^{\frac{k}{2}})$ as $\dot{H}^k(\Omega)$, which is a subspace of $H^k(\Omega)$ with norm $|\chi|_k = \|A^{\frac{k}{2}}\chi\|$.\\

\noindent
\textbf{Poincar\'e inequality}
\begin{equation}\label{poincare}
	\|u\|^2 \leq \frac{1}{\lambda_1} \|A^\half u\|^2 
\end{equation}
where $\lambda_1$ is the smallest positive eigenvalue of the Laplacian operator $A$.

The variational formulation of problem \eqref{se1.1}-\eqref{se1.2} is to find $u : [0, T] \rightarrow H_0^1(\Omega)$ for any $T > 0$ such that for all $v \in H^1_0(\Omega)$

\begin{eqnarray}\label{se1.4}
	(u'', v) + \beta~ a(u', v) + \alpha\,(u', v) + a(u, v) = 0,
\end{eqnarray}
with $u(0) = u_0 \; \mbox{and} \; u_t(0)=u_1 \; \mbox{in} \; \Omega$, where the bilinear form $a(\cdot,\cdot) : D(A^\half)\times D(A^\half) \to \mathbb{R}$ is defined as
\[
a(w,\chi) = \left(A^\half w,A^\half \chi\right).
\]

%%%%%%%%%%%%%%%%%%%%%%%%%%%%%%%%%%%%%%%%%%%%%%%%%%%%%%%%%%%%%%%%%%%%%%%%
%%%%%%%%%%%%%%%%%%%%%%%%%%%%%%%%%%%%%%%%%%%%%%%%%%%%%%%%%%%%%%%%%%%%%%%%		
\section{Continuous case}\label{sec:cont}
%%%%%%%%%%%%%%%%%%%%%%%%%%%%%%%%%%%%%%%%%%%%%%%%%%%%%%%%%%%%%%%%%%%%%%%%
%%%%%%%%%%%%%%%%%%%%%%%%%%%%%%%%%%%%%%%%%%%%%%%%%%%%%%%%%%%%%%%%%%%%%%%
This section discusses the decay of the energy $E(u)(t) = \half\left(\|u'\|^2 + | u|^2_1\right)$ based on a technique used in the exponential decay of the solution of the parabolic initial and boundary value problem.

By setting $u_\delta = e^{\delta t}u$, write 
\begin{equation*}\label{eq:u'delta}
	u_\delta' = \delta e^{\delta t}u + e^{\delta t}u' = \delta u_\delta + e^{\delta t}u'
\end{equation*}
and
\begin{equation*}\label{eq:u''delta}
	u_\delta'' = \delta^2u_\delta + 2\delta e^{\delta t}u' + e^{\delta t}u''.
\end{equation*}
 Now, rewrite \eqref{se1.4} as 
\begin{eqnarray}\label{se11.4}
	(u''_\delta, v) + \beta~ a(u'_\delta, v) + (\alpha-2\delta)(u'_\delta, v) + (1-\beta\delta)\,a( u_\delta, v) - \delta(\alpha-\delta)(u_\delta,v) = 0,
\end{eqnarray}
with $u_\delta(0) = u_0 \; \mbox{and} \; (u_\delta)'(0) = \delta u_0 + u_1 \; \mbox{in} \; \Omega.$

%%%%%%%%%%%%%%%%%%%%%%%%%%%%%%%%%%%%%%%%%%%%%%%%%%%%%%%%%%%%%%%%%%%%%%%%%%%%%
%%%%%%%%%%%%%%%%%%%%%%%%%%%%%%%%%%%%%%%%%%%%%%%%%%%%%%%%%%%%%%%%%%%%%%%%%%%%%
\begin{thm}\label{thm:eng}
 The following decay estimate hold for the energy $E(u)(t) = \half\left(\|u'\|^2 + |u|^2_1\right)$:
	\begin{enumerate}[label=(\roman*)]
		\item if $0 < \delta < \min\left(\frac{\alpha + \beta\lambda_1}{2}, \frac{\lambda_1}{\alpha + \beta\lambda_1}\right)$ or \, $0 < \delta = \frac{\alpha + \beta\lambda_1}{2}$, then
		\[
		E(u)(t) \leq C(\alpha,\beta,\lambda_1)e^{-2\delta t}E(u)(0),
		\]
		\item if $\delta = \frac{\lambda_1}{\alpha + \beta\lambda_1}$, then
		\[
		\|u'\|^2 \leq C e^{-2\delta t}E(u)(0).
			\]
	\end{enumerate} 
\end{thm}
	
\begin{proof}
	A choice of $v = u_\delta'$ in \eqref{se11.4} yields
	\begin{equation*}\label{eq:eqfordE}
		\half\frac{{\rm d}}{{\rm d}t}\left(\|u_\delta'\|^2 + (1 - \beta\delta)| u_\delta|^2_1 - ((\alpha - \delta)\delta)\|u_\delta\|^2\right) + (\alpha - 2\delta)\|u_\delta'\|^2 + \beta| u_\delta'|^2_1 = 0.
	\end{equation*}            
	A use of the Poincar\'e inequality \eqref{poincare} shows
	\begin{equation}\label{eq:eqfordEin}
		\half\frac{{\rm d}}{{\rm d}t}\left(\|u_\delta'\|^2 + (1 - \beta\delta)| u_\delta|^2_1 - ((\alpha - \delta)\delta)\|u_\delta\|^2\right) + \left(\alpha  + \beta\lambda_1 - 2\delta\right)\|u_\delta'\|^2 \leq 0.
	\end{equation}
	Setting 
	\begin{equation}\label{eq:eude}
		\tilde{E}(u_\delta)(t) = \half\left(\|u_\delta'\|^2 + (1 - \beta\delta)| u_\delta|^2_1 - ((\alpha - \delta)\delta)\|u_\delta\|^2\right)
	\end{equation} 
	 an integration of equation \eqref{eq:eqfordEin} in time with $\delta \leq \frac{\alpha + \beta\lambda_1}{2}$ reveals
	\begin{equation}\label{eq:Etildere}
		\tilde{E}(u_\delta)(t) \leq \tilde{E}(u_\delta)(0) \leq C E(u)(0). %= \half\left(\|u_1\|^2 + (1 - \beta\delta)\|A^\half u_0\|^2 - ((\alpha - \delta)\delta)\|u_0\|^2\right).
	\end{equation}
	Again, an application of \eqref{poincare} with $\delta < \frac{\lambda_1}{\alpha + \beta\lambda_1}$ and $\gamma_0 = 1 - \left(\frac{\alpha + \beta\lambda_1}{\lambda_1}\right)\delta$ shows
	\begin{equation*}
		\begin{aligned}
			\tilde{E}(u_\delta)(t) &= \half\left(\|u_\delta'\|^2 + (1 - \beta\delta)| u_\delta|^2_1 - \alpha\delta\|u_\delta\|^2 + \delta^2\|u_\delta\|^2\right)
			\\
			&\geq \half\left(\|\delta u_\delta + e^{\delta t}u'\|^2 + \frac{1}{\lambda_1}\left( \lambda_1 - \left(\alpha + \beta\lambda_1\right)\delta\right)| u_\delta|^2_1\right) + \frac{\delta^2}{2}\|u_\delta\|^2
			\\
			&\geq \frac{e^{2\delta t}}{2}\min\left(1,\gamma_0\right)\left(\|u'\|^2 + | u|^2_1\right)
			\\
			&= \gamma_0e^{2\delta t}E(u)(t).
		\end{aligned}
	\end{equation*}
On substituting in \eqref{eq:Etildere}, we conclude the proof of $(i)$. \\
	For $(ii)$, if $\delta = \frac{\lambda_1}{\alpha + \beta\lambda_1}$, then  a use of \eqref{poincare} in \eqref{eq:eude} yields
	\begin{equation}\label{eq:Edelta}
		\begin{aligned}
			\tilde{E}(u_\delta)(t) &\geq \half\left(\|u'_\delta\|^2 + \delta^2\|u_\delta\|^2\right) = \half\left(\|e^{2\delta t} u + \delta u_\delta\|^2 + \delta^2\|u_\delta\|^2\right)
			\\
			&\geq \half e^{2\delta t}\|u'\|^2.
		\end{aligned}
	\end{equation} 
Again, substitute \eqref{eq:Edelta} in \eqref{eq:Etildere} to complete the rest of the proof. 
\end{proof}

%%%%%%%%%%%%%%%%%%%%%%%%%%%%%%%%%%%%%%%%%%%%%%%%%%%%%%%%%%%%%%%%%%%%%%%%%%%
%%%%%%%%%%%%%%%%%%%%%%%%%%%%%%%%%%%%%%%%%%%%%%%%%%%%%%%%%%%%%%%%%%%%%%%%%%%	
Next, we study the decay behaviour of the higher order energy. By taking $k - 1 \geq 0$ times derivative of the \eqref{se1.4} with respect to time variable, we obtain
\begin{equation}\label{se1.4k}
	\left(u^{(k+1)}, v\right) + \beta~a( u^{(k)}, v) + \alpha\left(u^{(k)}, v\right) + a(u^{(k-1)}, v) = 0, \quad \forall \quad v \in H^1_0(\Omega).
\end{equation}
With $u_\delta^{(k-1)} = e^{\delta t}u^{(k-1)}$,  write

\begin{equation*}\label{eq:u'deltak}
	u_\delta^{(k)} = \delta e^{\delta t}u^{(k-1)} + e^{\delta t}u^{(k)} = \delta u_\delta^{(k-1)} + e^{\delta t}u^{(k)}
\end{equation*}
and
\begin{equation*}\label{eq:u''deltak}
	u_\delta^{(k+1)} = \delta^2u_\delta^{(k-1)} + 2\delta e^{\delta t}u^{(k)} + e^{\delta t}u^{(k+1)}.
\end{equation*}
The variational formulation \eqref{se1.4k} becomes: find $u_\delta^{(k)} : [0, T] \rightarrow H_0^1(\Omega)$ for any $T > 0$ such that for all $v \in H_0^1(\Omega)$
\begin{equation}\label{se11.4k}
	\begin{aligned}
		(u^{(k+1)}_\delta, v) + \beta~a(u^{(k)}_\delta, v) + (\alpha-2\delta)(u^{(k)}_\delta, v) &+ (1-\beta\delta)~a(u^{(k-1)}_\delta, v) 
		\\
		& - \delta(\alpha-\delta)(u_\delta^{(k-1)},v) = 0.
	\end{aligned}
\end{equation}

%%%%%%%%%%%%%%%%%%%%%%%%%%%%%%%%%%%%%%%%%%%%%%%%%%%%%%%%%%%%%%%%%%%%%%%%%%%%%%

\begin{thm}\label{thm:higheng}
	For $0 < \delta \leq \min\left(\frac{\alpha + \beta\lambda_1}{2}, \frac{\lambda_1}{\alpha + \beta\lambda_1}\right)$ and $k = 1,2,3,\dots$, the solution $u$ of \eqref{se1.1} satisfies
	\begin{enumerate}[label=(\roman*)]
		\item if $0 < \delta < \min\left(\frac{\alpha + \beta\lambda_1}{2}, \frac{\lambda_1}{\alpha + \beta\lambda_1}\right)$ or \, $0 < \delta = \frac{\alpha + \beta\lambda_1}{2}$, then
		\[
		  E^{(k)}(u)(t) \leq C(\alpha,\beta,\lambda_1)e^{-2\delta t}E^{(k)}(u)(0),
		\] 
		where $E^{(k)}(u)(t) = \half\left(\|u^{(k)}\|^2 + |u^{(k-1)}|^2_1\right)$, and $u^{(k)}$ is the $k^{th}$ time derivative of $u$.
		\item if $0 < \delta = \frac{\lambda_1}{\alpha + \beta\lambda_1}$, then
		\[
		 \|u^{(k)}\|^2 \leq Ce^{-2\delta t}E^{(k)}(u)(0).
		\]
	\end{enumerate}
\end{thm} 
\begin{proof}
	Set $u_\delta^{(k-1)} = w_\delta$, then we obtain from \eqref{se11.4k}
	\begin{equation*}
		\begin{aligned}
			(w''_\delta, v) + \beta~a(w'_\delta, v) + (\alpha-2\delta)(w'_\delta, v) &+ (1-\beta\delta)~a(w_\delta, v) 			
			- \delta(\alpha-\delta)(w_\delta,v) = 0,
		\end{aligned}
	\end{equation*}
	which corresponds to \eqref{se11.4} after replacing $u_\delta$ by $w_\delta$. Therefore, following the arguments in the proof of Theorem \ref{thm:eng}, we conclude the result for $w_\delta$. This completes the rest of the proof. 
\end{proof}
%%%%%%%%%%%%%%%%%%%%%%%%%%%%%%%%%%%%%%%%%%%%%%%%%%%%%%%%%%%%%%%%%%%%%%%%%%%%%%%
%%%%%%%%%%%%%%%%%%%%%%%%%%%%%%%%%%%%%%%%%%%%%%%%%%%%%%%%%%%%%%%%%%%%%%%%%%%%%%%
\begin{thm}\label{thm:Aeng}
	For $0 < \delta \leq \min\left(\frac{\alpha + \beta\lambda_1}{2}, \frac{\lambda_1}{\alpha+\beta\lambda_1}\right)$, the solution $u$ of \eqref{se1.1} satisfies
	\begin{enumerate}[label=(\roman*)]
		\item if $0 < \delta < \min\left(\frac{\alpha + \beta\lambda_1}{2}, \frac{\lambda_1}{\alpha+\beta\lambda_1}\right)$ or \, $0 < \delta = \frac{\alpha + \beta\lambda_1}{2}$, then
		\[
		E_A(u)(t) \leq C(\alpha,\beta,\lambda_1)e^{-2\delta t}E_A(u)(0),
		\]
		where, $E_A(u)(t) = \half\left(|u'|^2_1 + |u|^2_2\right)$.
		\item If $0 < \delta = \frac{\lambda_1}{\alpha + \beta\lambda_1}$, then
		\[
		|u'|^2_1 \leq C e^{-2\delta t}E_A(u)(0).
		\]
	\end{enumerate}
	
\end{thm}
\begin{proof}
	Setting $v = Au_\delta'$ in \eqref{se11.4} and using integration by part, we obtain
	\begin{equation*}
		\half\frac{{\rm d}}{{\rm d}t}\left(|u_\delta'|^2_1 + \left(1 - \beta\delta\right)|u_\delta|^2_2 - \delta\left(\alpha - \delta\right)|u_\delta|^2_1\right) + \beta|u_\delta'|^2_2 + \left(\alpha - 2\delta\right)|u_\delta'|^2_1 = 0.
	\end{equation*}
	Now, using the same technique as in Theorem \ref{thm:eng}, we obtain the result. This completes the proof. 
\end{proof}
%%%%%%%%%%%%%%%%%%%%%%%%%%%%%%%%%%%%%%%%%%%%%%%%%%%%%%%%%%%%%%%%%%%%%%%%%%%%%%%%%%
%%%%%%%%%%%%%%%%%%%%%%%%%%%%%%%%%%%%%%%%%%%%%%%%%%%%%%%%%%%%%%%%%%%%%%%%%%%%%%%%%%
\begin{thm}\label{thm:geno}
    Assume that $u_0 \in D(A^{(j)})$ and $u_1 \in D(A^{(j-\half)})$ for $j > 1$. Then, using the techniques of Theorem \ref{thm:higheng}-\ref{thm:Aeng} and the induction, there holds
    \begin{enumerate}[label=(\roman*)]
    	\item if $0 < \delta < \min\left(\frac{\alpha + \beta\lambda_1}{2}, \frac{\lambda_1}{\alpha+\beta\lambda_1}\right)$ or \, $0 < \delta = \frac{\alpha + \beta\lambda_1}{2}$, then
		\[
		  E^{(k)}_{A^{(j)}}(u)(t) \leq C(\alpha,\beta,\lambda_1)e^{-2\delta t}E^{(k)}_{A^{(j)}}(u)(0),
		\]         
		where, $E^{(k)}_{A^{(j)}}(u)(t) = \half\left(|u^{(k)}(t)|^2_{2j-1} + |u^{(k-1)}(t)|^2_{2j}\right)$.
		\item If $0 < \delta = \frac{\lambda_1}{\alpha + \beta\lambda_1}$, then
		\[
		 |u^{(k)}(t)|^2_{2j-1} \leq C e^{-2\delta t}E^{(k)}_{A^{(j)}}(u)(0).
		\]
	\end{enumerate}	
\end{thm}	

	\begin{Rem}
		We observe from Theorem \ref{thm:geno} that the constant $C = C(\alpha,\beta,\lambda_1)$ does not blow up when $\alpha$ and $\beta$ tend to zero. Therefore, when $\alpha \to 0, \beta > 0$, the decay rate in Theorem \ref{thm:geno} becomes $0 < \delta < \min\left(\frac{\beta\lambda_1}{2}, \frac{1}{\beta}\right)$.
		In case $\beta \to 0, \alpha > 0$, the range of decay in Theorem \ref{thm:geno} becomes $0 < \delta < \min\left(\frac{\alpha}{2}, \frac{\lambda_1}{\alpha}\right)$, which is better than the decay estimate derived in \cite{danumjaya2023asymptotic}.
	\end{Rem}
	\begin{Rem}
		Let $u^{\alpha,\beta}$ be the solution of \eqref{se1.1} and let $u^{0,\beta}$ and $u^{\alpha,0}$ be the solution of the problems with $\alpha = 0$ and $\beta = 0$, respectively, then by using similar argument as in the proof of Theorem \ref{thm:geno}, it is easy to show
		\[
		E(u^{\alpha,\beta} - u^{0,\beta})(t) = O(\alpha e^{-2\delta_0 t}),
		\]
		and
		\[
		E(u^{\alpha,\beta} - u^{\alpha,0})(t) = O(\beta e^{-2\delta_1 t}),
		\] 
		provided $0 < \delta_0 < \min\left(\frac{\beta\lambda_1}{2}, \frac{\lambda_1}{\alpha + \beta\lambda_1}\right)$ and $0 < \delta_1 < \min\left(\frac{\alpha}{2}, \frac{\lambda_1}{\alpha + \beta\lambda_1}\right)$, respectively.
\end{Rem}

%%%%%%%%%%%%%%%%%%%%%%%%%%%%%%%%%%%%%%%%%%%%%%%%%%%%%%%%%%%%%%%
%%%%%%%%%%%%%%%%%%%%%%%%%%%%%%%%%%%%%%%%%%%%%%%%%%%%%%%%%%%%%%%	
\section{Semidiscrete Method}\label{sec:semd}
%%%%%%%%%%%%%%%%%%%%%%%%%%%%%%%%%%%%%%%%%%%%%%%%%%%%%%%%%%%%%%%
%%%%%%%%%%%%%%%%%%%%%%%%%%%%%%%%%%%%%%%%%%%%%%%%%%%%%%%%%%%%%%%

Let $\mathcal{T}_h$ be a shape regular triangulation of $\bar{\Omega}$ and let $h = \max \mbox{diam}(K) ~\forall ~K \in \mathcal{T}_h$ be the discretization parameter. Let $X_h$ be the finite-dimensional subspace of $H_0^1$ with the following approximation property:
\begin{equation*}
	\inf_{\psi_h \in X_h}\left\{\|\psi - \psi_h\| + h\|\psi - \psi_h\|_1\right\} \leq Ch^2\|\psi\|_2, ~ \forall ~ \psi \in H^2\cap H^1_0.
\end{equation*} 
Set the discrete bilinear form by 
\[
a_h(w_h,\chi_h) = \left(A_h^\half w_h, A_h^\half\chi_h\right) = \left(A_hw_h, \chi_h\right), ~ \forall~ w_h,\chi_h \in X_h,
\] 
 with the discrete norm by $|\chi_h|_{1,h} = \|A_h^{\frac{1}{2}}\chi_h\|,$ which is actually $|\chi_h|_1$ and $\|\chi_h\|_{2,h}= \|A_h \chi_h\|.$

The semidiscrete formulation is defined as: Find $u_h : [0,\infty) \to X_h$ such that
\begin{equation}\label{sdf}
	\begin{aligned}
		\left(u_h'', v\right) + \beta~a_h(u_h', v) &+  \alpha\left(u_h', v\right) + a_h(u_h, v) = 0, ~ \forall ~ v \in X_h,
	\end{aligned}
\end{equation}
with $u_h(0) = u_{0h} ~ \mbox{and}~ u_h'(0) = u_{1h}$ to be defined later.

By defining $u_{h,\delta}(t) = e^{\delta t}u_h$, we can rewrite our semidiscrete formulation as
\begin{equation}\label{deltasdf}
	\begin{aligned}
		\left(u_{h,\delta}'', v\right) + \beta~a_h(u_{h,\delta}', v) + (\alpha - 2\delta)\left(u_{h,\delta}', v\right)& + (1 - \beta\delta)~a_h(u_{h,\delta}, v)
		\\
		& - \delta(\alpha - \delta)\left(u_{h,\delta}, v\right) = 0, ~\forall ~ v \in X_h.
	\end{aligned}
\end{equation}     
The discrete energy is given by 
\[
E_h(u_h)(t) = \half\left(\|u_h'\|^2 + |u_h|^2_1\right).
\]
%%%%%%%%%%%%%%%%%%%%%%%%%%%%%%%%%%%%%%%%%%%%%%%%%%%%%%%%%%%%%%%%%%%%%%%%%%%%%%
%%%%%%%%%%%%%%%%%%%%%%%%%%%%%%%%%%%%%%%%%%%%%%%%%%%%%%%%%%%%%%%%%%%%%%%%%%%%%%
\begin{thm}\label{thm:diseng}
	For $0 < \delta \leq \min\left(\frac{\alpha + \beta\lambda_1}{2}, \frac{\lambda_1}{\alpha + \beta\lambda_1}\right)$, the solution $u_h$ of \eqref{sdf} satisfies
	\begin{enumerate}[label=(\roman*)]
		\item if $0 < \delta < \min\left(\frac{\alpha + \beta\lambda_1}{2}, \frac{\lambda_1}{\alpha + \beta\lambda_1}\right)$ or \, $0 < \delta = \frac{\alpha + \beta\lambda_1}{2}$, then
		\[
		E_h(u_h)(t) \leq C(\alpha,\beta,\lambda_1)e^{-2\delta t}E_h(u_h)(0),
		\]
		\item if $0 < \delta = \frac{\lambda_1}{\alpha + \beta\lambda_1}$, then
		\[
		\|u_h'(t)\|^2 \leq C e^{-2\delta t}E_h(u_h)(0).
		\]
	\end{enumerate}
	 
\end{thm}

\begin{proof}
By choosing $v = u_{h,\delta}'$ in \eqref{deltasdf}, we obtain
\begin{equation*}
	\begin{aligned}
		\half\frac{{\rm d}}{{\rm d}t}\left(\|u_{h,\delta}'\|^2 + \left(1 - \beta \delta\right)|u_{h,\delta}|^2_1 - \delta(\alpha - \delta)\|u_{h,\delta}\|^2\right) + \beta|u_{h,\delta}'|^2_1 + (\alpha - 2 \delta)\|u_{h,\delta}'\|^2 = 0.
	\end{aligned}
\end{equation*}
Since $u_h \in X_h \subset H_0^1$, then by the Poincar\'e inequality \eqref{poincare}, we have
\begin{equation}\label{dispoincare}
	\|u_h(t)\| \leq \frac{1}{\sqrt{\lambda_1}}|u_h(t)|^2_1.
\end{equation}
By following steps similar to the proof of Theorem \ref{thm:eng} and replacing $u$ by $u_h$ and $u_\delta$ by $u_{h,\delta}$, we obtain the desired result.
\end{proof}

%%%%%%%%%%%%%%%%%%%%%%%%%%%%%%%%%%%%%%%%%%%%%%%%%%%%%%%%%%%%%%%%%%
%%%%%%%%%%%%%%%%%%%%%%%%%%%%%%%%%%%%%%%%%%%%%%%%%%%%%%%%%%%%%%%%%%

\begin{thm}\label{thm:dishigheng}
	For $0 < \delta \leq \min\left(\frac{\alpha + \beta\lambda_1}{2}, \frac{\lambda_1}{\alpha + \beta\lambda_1}\right)$ and $k = 1,2,3,\dots$, the solution $u_h$ of \eqref{sdf} satisfies
	\begin{enumerate}[label=(\roman*)]
		\item if $0 < \delta < \min\left(\frac{\alpha + \beta\lambda_1}{2}, \frac{\lambda_1}{\alpha + \beta\lambda_1}\right)$ or \, $0 < \delta = \frac{\alpha + \beta\lambda_1}{2}$, then
		\[
		E^{(k)}_h(u_h)(t) \leq C(\alpha,\beta,\lambda_1)e^{-2\delta t}E_h^{(k)}(u_h)(0),
		\]
		\item if $0 < \delta = \frac{\lambda_1}{\alpha + \beta\lambda_1}$, then
		\[
		\|u^{(k)}_h(t)\|^2 \leq C e^{-2\delta t}E_h^{(k)}(u_h)(0),
		\]
		\end{enumerate}
	where $E^{(k)}_h(u_h)(t) = \half\left(\|u^{(k)}_h\|^2 + |u^{(k-1)}_h|^2_1\right)$.
\end{thm}
\begin{proof}
	The proof is similar to the proof of Theorem \ref{thm:higheng}.
\end{proof}

%%%%%%%%%%%%%%%%%%%%%%%%%%%%%%%%%%%%%%%%%%%%%%%%%%%%%%%%%%%%%%%%%%%%%%%%%%%%
%%%%%%%%%%%%%%%%%%%%%%%%%%%%%%%%%%%%%%%%%%%%%%%%%%%%%%%%%%%%%%%%%%%%%%%%%%%%

\begin{thm}\label{thm:disAeng}
	For $0 < \delta \leq \min\left(\frac{\alpha + \beta\lambda_1}{2},\frac{\lambda_1}{\alpha + \beta\lambda_1}\right)$, the solution $u_h$ of \eqref{sdf} satisfies
	\begin{enumerate}[label=(\roman*)]
		\item if $0 < \delta < \min\left(\frac{\alpha + \beta\lambda_1}{2}, \frac{\lambda_1}{\alpha + \beta\lambda_1}\right)$ or \, $0 < \delta = \frac{\alpha + \beta\lambda_1}{2}$, then
		\[
		E_{(A_h)}(u_h)(t) \leq C(\alpha,\beta,\lambda_1)e^{-2\delta t}E_{(A_h)}(u_h)(0), ~ {\color{red} \forall~ t \geq 0},
		\]
		\item if $0 < \delta = \frac{\lambda_1}{\alpha + \beta\lambda_1}$, then
		\[
		|u_h'(t)|^2_1 \leq C e^{-2\delta t}E_{(A_h)}(u_h)(0),
		\]
	\end{enumerate}
	where $E_{(A_h)}(u_h)(t) = \half\left(|u_h'|^2_1 + |u_h|^2_{2,h}\right)$.
\end{thm}

\begin{proof}
	Choosing $v = A_hu_{h,\delta}'$ in \eqref{deltasdf} and using definition of $A_h$, we obtain
	\begin{equation*}
		\begin{aligned}
			\half\frac{{\rm d}}{{\rm d}t}\left(|u_{h,\delta}'|^2_1 + \left(1 - \beta\delta\right)|u_{h,\delta}|^2_{2,h} - \delta\left(\alpha-\delta\right)|u_{h,\delta}|^2_1\right) &+ \beta|u_{h,\delta}'|^2_{2,h} 
			+ \left(\alpha - 2\delta\right)|u_{h,\delta}'|^2_1 = 0.
		\end{aligned}
	\end{equation*}
	Proceeding in the similar manner like the proof of Theorem \ref{thm:Aeng} with the use of the Poincar\'e inequality \eqref{dispoincare} give the desired result. This completes the rest of the proof.  
\end{proof}

%%%%%%%%%%%%%%%%%%%%%%%%%%%%%%%%%%%%%%%%%%%%%%%%%%%%%%%%%%%%%%%%
\subsection{Error estimates}
%%%%%%%%%%%%%%%%%%%%%%%%%%%%%%%%%%%%%%%%%%%%%%%%%%%%%%%%%%%%%%%%
This subsection is on error estimates with decay property of the semidiscrete problem.

Let $\Pi_hu$ be the elliptic projection of $u$ defined by
\[
a_h((u - \Pi_hu), \chi) = 0, ~ \forall ~ \chi \in X_h.
\]
Let 
\[
e := u - u_h = (u - \Pi_hu) + (\Pi_hu - u_h) = \eta + \theta,
\]
and
\[
e_\delta := u_\delta - u_{h,\delta} = (u_\delta - \Pi_hu_\delta) + (\Pi_hu_\delta - u_{h,\delta}) = \eta_\delta + \theta_\delta.
\]
As $\Pi_hu$ is the elliptic projection, therefore,
\[
a_h(\eta'_\delta, v) = \delta e^{\delta t}a_h(\eta, v) + e^{\delta t}a_h((u' - \Pi_hu'), v) = 0, ~\forall ~ v \in X_h.
\] 
The projection operator satisfies the following approximation properties
\begin{equation*}\label{eq:projest}
	\|\eta\|_j + \|\eta'\|_j + \|\eta''\|_j \leq Ch^{r+1-j}\left(\sum_{m=0}^{2}\Big\Vert\frac{\partial^m u}{\partial t^m}\Big\Vert_{r+1}\right), ~ j = 0,1.
\end{equation*}
\begin{equation*}
	\|\eta_\delta\|_j + \|\eta'_\delta\|_j + \|\eta''_\delta\|_j \leq Ch^{r+1-j}e^{\delta t}\left(\sum_{m=0}^{2}\Big\Vert\frac{\partial^m u}{\partial t^m}\Big\Vert_{r+1}\right), ~ j = 0,1.
\end{equation*}
Set
	\begin{equation*}\label{eq:deltaetarelation}
		\begin{aligned}
			\eta_\delta = e^{\delta t}\eta, \;	\eta_\delta' = e^{\delta t}(\delta \eta + \eta') \; \mbox{and} \; \eta_\delta'' = e^{\delta t}(\eta'' + 2\delta\eta' + \delta^2\eta).
		\end{aligned}
	\end{equation*}
By subtracting equation \eqref{sdf} from equation \eqref{se1.4} and using the definition of the elliptic projection, we arrive at the following
\begin{equation}\label{eq:error}
	\begin{aligned}
		\left(\theta'', v\right) + \beta~a_h(\theta', v) + \alpha\left(\theta', v\right) + a_h(\theta, v) = -\left(\eta'', v\right) - \alpha\left(\eta', v\right), ~\forall~ v \in X_h.
	\end{aligned}
\end{equation} 
Similarly, we obtain 
\begin{equation}\label{eq:deltaerror}
	\begin{aligned}
		\left(\theta_\delta'', v\right) + \beta~a_h(\theta_\delta', v) &+ (\alpha - 2\delta)\left(\theta_\delta', v\right) + (1 - \beta\delta)~a_h(\theta_\delta, v) 
		\\
		& - \delta(\alpha - \delta)\left(\theta_\delta, v\right) 
		=-e^{\delta t}\left(\eta'', v\right) - \alpha e^{\delta t}\left(\eta', v\right), ~ \forall ~ v \in X_h.
	\end{aligned}
\end{equation}

%%%%%%%%%%%%%%%%%%%%%%%%%%%%%%%%%%%%%%%%%%%%%%%%%%%%%%%%%%%%%%%%%%%%%
%%%%%%%%%%%%%%%%%%%%%%%%%%%%%%%%%%%%%%%%%%%%%%%%%%%%%%%%%%%%%%%%%%%%%

\begin{thm}\label{thm:thetaetarel}
	Let $\theta$ satisfy \eqref{eq:error}. Then, for small $\delta_0>0$ and $0 < \delta \leq \min\left(\frac{\alpha + \beta\lambda_1-\delta_0\lambda_1}{2}, \frac{\lambda_1}{\alpha + \beta\lambda_1}\right)$ there holds
	\begin{enumerate}[label=(\roman*)]
		\item if $0 < \delta < \min\left(\frac{\alpha + \beta\lambda_1-\delta_0\lambda_1}{2}, \frac{\lambda_1}{\alpha + \beta\lambda_1}\right)$ or \, $0 < \delta = \frac{\alpha + \beta\lambda_1-\delta_0\lambda_1}{2}$, then
		\[
		E_h(\theta)(t) \leq Ce^{-2\delta t}E_h(\theta)(0) + C\int_0^t e^{-2\delta(t-s)}\left(\|\eta''(s)\|^2 + \|\eta'(s)\|^2 \right){\rm d}s, \quad \forall~ t \geq 0,
		\]
		\item if $0 < \delta = \frac{\lambda_1}{\alpha + \beta\lambda_1}$, then
		\[
		\|\theta'(t)\|^2 \leq Ce^{-2\delta t}E_h(\theta)(0) + C\int_0^t e^{-2\delta(t-s)}\left(\|\eta''(s)\|^2 + \|\eta'(s)\|^2 \right){\rm d}s.
		\]
	\end{enumerate}
\end{thm}

\begin{proof}
	A use of $v = \theta_\delta'$ in \eqref{eq:deltaerror} yields
	\begin{equation*}
		\begin{aligned}
			\half\frac{{\rm d}}{{\rm d}t}\left(\|\theta_\delta'\|^2 + (1 - \beta\delta)|\theta_\delta|^2_1 - \delta(\alpha - \delta)\|\theta_\delta\|^2\right) &+ \beta|\theta_\delta'|^2_1 + (\alpha - 2\delta)\|\theta_\delta'\|^2 
			\\
			&=-e^{\delta t}\left(\eta'', \theta_\delta'\right) - \alpha e^{\delta t}\left(\eta', \theta_\delta'\right).
		\end{aligned}
	\end{equation*}
For some $\delta_0 > 0$, a use of the H\"older inequality with the Young's inequality shows
	\begin{equation*}
		\frac{{\rm d}}{{\rm d}t}\tilde{E}_h(\theta_\delta)(t) + (\alpha + \beta\lambda_1 - 2\delta)\|\theta_\delta'\|^2 \leq \frac{4e^{2\delta t}}{\delta_0\lambda_1}\left(\|\eta''\|^2 + \alpha^2\|\eta'\|^2\right) + \delta_0\lambda_1\|\theta_\delta'\|^2.
	\end{equation*}
	Using kick-back argument with $\delta \leq \frac{\alpha + \beta\lambda_1-\delta_0\lambda_1}{2}$ and integrating in time from $0$ to $t$, we obtain
	\begin{equation*}
		\begin{aligned}
			\tilde{E}_h(\theta_\delta)(t) \leq \tilde{E}_h(\theta_\delta)(0) + \frac{4}{\delta_0\lambda_1}\int_0^te^{2\delta s}\left(\|\eta''(s)\|^2 + \alpha^2\|\eta'(s)\|^2 \right){\rm d}s.
		\end{aligned}
	\end{equation*}
Rest of the proof follows by using an argument similar to the proof of Theorem \ref{thm:eng}.
\end{proof}	

%%%%%%%%%%%%%%%%%%%%%%%%%%%%%%%%%%%%%%%%%%%%%%%%%%%%%%%%%%%%%%%%%%%%%%%%
%%%%%%%%%%%%%%%%%%%%%%%%%%%%%%%%%%%%%%%%%%%%%%%%%%%%%%%%%%%%%%%%%%%%%%%%
\noindent
Note that $\theta(0) = 0$ if $u_{0h} = \Pi_hu_0$ and, therefore,
\[
E_h(\theta)(0) = \half\|\theta'(0)\|^2.
\]
Let $u_{1h}$ be either the $L^2$-projection or the interpolant of $u_1$ in $X_h$, then
\begin{equation*}
	E_h(\theta)(0) \leq Ch^4\|u_1\|^2_2.
\end{equation*}
Hence, from Theorem \ref{thm:thetaetarel}, we obtain
\begin{equation}\label{eq:2.37}
	\|\theta'(t)\|^2 +|\theta(t)|^2_1\leq Ch^4e^{-2\delta t}\left(\|u_1\|^2_2 + \int_0^te^{2\delta s}\left(\|u''(s)\|^2_2 + \|u'(s)\|^2_2 \right){\rm d}s\right).
\end{equation}
In order to estimate the two terms under the integral sign on the right-hand side of \eqref{eq:2.37}, we note from Theorem \ref{thm:geno}, when $j = \frac{3}{2}$ and $k = 2$ that
\begin{equation}\label{eq:2.40}
	\begin{aligned}
		\|u''(t)\|^2_2 \leq Ce^{-2\delta t}E_{A^{(\frac{3}{2})}}^{(2)}(u)(0) &\leq Ce^{-2\delta t}\left(| u^{(2)}(0)|^2_2 
        + |u'(0)|^2_3\right)
		\\
		&\leq Ce^{-2\delta t}\left(\beta\|u_1\|_4^2 + \|u_1\|_3^2 + \|u_0\|_4^2\right).
	\end{aligned}
\end{equation}
Here, we have used from equation \eqref{se1.1}
\begin{equation}
	u''(0) = -\beta Au'(0) - \alpha u'(0) - Au(0) = -\beta Au_1 - \alpha u_1 - Au_0
\end{equation}
and it's estimate. For $\|u'(t)\|_2$ estimate, apply the Theorem \ref{thm:geno} for $j = \frac{3}{2}$ and  $k = 1$ to arrive at
\begin{equation}\label{eq:2.41}
	\begin{aligned}
		\|u'(t)\|^2_2 \leq Ce^{-2\delta t}E_{A^{(\frac{3}{2})}}^{(1)}(u)(0) & \leq Ce^{-2\delta t}\left(| u'(0)|^2_2 + |u_0|^2_3\right)
		\\
		&\leq Ce^{-2\delta t}\left(\|u_1\|^2_2 + \|u_0\|^2_3\right).
	\end{aligned}
\end{equation}
A use of equations \eqref{eq:2.40}-\eqref{eq:2.41} in equation \eqref{eq:2.37}, shows
\begin{equation*}
	\|\theta'(t)\|^2 + |\theta(t)|_1^2\leq Ch^4(1 + t)e^{-2\delta t}\left(\beta\|u_1\|_4^2 + \|u_1\|_3^2 + \|u_0\|_4^2\right).
\end{equation*}
Hence, a use of triangle inequality with estimates of $\eta$ implies
\begin{equation}\label{estimate:super-cgt}
	\|u'(t) - u_h'(t)\|^2  + |u(t)-u_h(t)|_1^2 \leq Ch^4(1+t)e^{-2\delta t}\left(\beta\|u_1\|_4^2 + \|u_1\|_3^2 + \|u_0\|_4^2\right).
\end{equation}

%%%%%%%%%%%%%%%%%%%%%%%%%%%%%%%%%%%%%%%%%%%%%%%%%%%%%%%%%%%%%%%%%%%%%%%%%%%%%
\begin{Rem}
	If $\beta \to 0$, then
	\begin{equation*}
		\|u'(t) - u_h'(t)\|^2+|u(t)-u_h(t)|_1^2 \leq Ch^4(1+t)e^{-2\delta t}\left(\|u_1\|_3^2 + \|u_0\|_4^2\right).
	\end{equation*}
\end{Rem}
%%%%%%%%%%%%%%%%%%%%%%%%%%%%%%%%%%%%%%%%%%%%%%%%%%%%%%%%%%%%%%%%%%%%%%%%%%%%

Below, the following result on optimal estimate in $H^1$-norm is given in terms of a theorem.
\begin{thm}\label{thm:H1est}
	Let $u_{0h}$ and $u_{1h}$ be either a $L^2$-projection or an interpolant of $u_0$ and $u_1$, respectively. Then, for small $\delta_0 > 0$ and for $0 < \delta \leq \min\left(\frac{\alpha + \beta\lambda_1-\delta_0\lambda_1}{2}, \frac{\lambda_1}{\alpha + \beta\lambda_1}\right),$ the following  optimal estimate holds
	\[
	\|\nabla(u - u_h)(t)\|^2 \leq Ch^2(1 + t)e^{-2\delta t}\left(\|u_0\|^2_3 + \beta\|u_1\|^2_3 + \|u_1\|^2_2\right).
	\]
	If $\beta \to 0$, then
	\[
	\|\nabla(u - u_h)(t)\|^2 \leq Ch^2(1 + t)e^{-2\delta t}\left(\|u_0\|^2_3 + \|u_1\|^2_2\right).
	\]
\end{thm}
\begin{proof}
Since $u_{0h}$ and $u_{1h}$ are either  $L^2$-projection or interpolant of $u_0$ and $u_1$, respectively, then
\begin{equation*}
	E_h(\theta)(0) \leq Ch^2\left(\|u_0\|^2_2 + \|u_1\|^2_1\right),
\end{equation*} 
and hence from Theorem \ref{thm:thetaetarel}, for $0 < \delta < \min\left(\frac{\alpha + \beta\lambda_1-\delta_0\lambda_1}{2}, \frac{\lambda_1}{\alpha + \beta\lambda_1}\right)$ or $0 < \delta = \frac{\alpha + \beta\lambda_1 - \delta_0\lambda_1}{2}$
\begin{equation}
	\begin{aligned}
		\|\nabla \theta\|^2 &\leq \|A_h^\half\theta\|^2 \leq Ce^{-2\delta t}\left(E_h(\theta)(0) + \int_0^te^{2\delta s}\left(\|\eta''(s)\|^2 + \|\eta'(s)\|^2\right){\rm d}s\right)
		\\
		&\leq Ch^2e^{-2\delta t}\left(\|u_0\|^2_2 + \|u_1\|^2_1 + \int_0^te^{2\delta s}\left(\|u''(s)\|^2_1 + \|u'(s)\|^2_1 \right){\rm d}s\right).\label{estimate-H1}
	\end{aligned}
\end{equation}
A use of Theorem \ref{thm:geno} yields
\begin{equation}\label{estimate:u-2}
	\begin{aligned}
		\|u''(s)\|^2_1 &\leq C\|A^\half u''(s)\|^2 \leq Ce^{-2\delta s}E^{(2)}_{A}(u)(0) \leq Ce^{-2\delta s}\left(\beta\|u_1\|^2_3 + \|u_1\|^2_2 + \|u_0\|^2_3\right)
	\end{aligned}
    \end{equation}
    and 
    \begin{equation}\label{estimate:u-1}
	\begin{aligned}
		\|u'(s)\|^2_1 &\leq C\|A^\half u'(s)\|^2 \leq Ce^{-2\delta s}E^{(1)}_{A}(u)(0) \leq Ce^{-2\delta s}\left(\|u_1\|^2_1 + \|u_0\|^2_2\right)
%		\\
%		\|u(s)\|^2_1 &\leq C\|u(s)\|^2_2 \leq C\|Au(s)\|^2 \leq Ce^{-2\delta s}E^{(1)}_{A^{(1)}}(u)(0) \leq Ce^{-2\delta s}\left(\|u_1\|^2_1 + \|u_0\|^2_1\right).
	\end{aligned}
\end{equation}
Substitute \eqref{estimate:u-2} and \eqref{estimate:u-1} in \eqref{estimate-H1}  to complete the rest of the proof.
\end{proof}
%%%%%%%%%%%%%%%%%%%%%%%%%%%%%%%%%%%%%%%%%%%%%%%%%%%%%%%%%%%%%%%%%%%%%%%%%%%%%

%%%%%%%%%%%%%%%%%%%%%%%%%%%%%%%%%%%%%%%%%%%%%%%%%%%%%%%%%%%%%%%%%%%%%%%%%%%%%
%%%%%%%%%%%%%%%%%%%%%%%%%%%%%%%%%%%%%%%%%%%%%%%%%%%%%%%%%%%%%%%%%%%%%%%%%%%%%
As a consequence of super-convergence result in \eqref{estimate:super-cgt}, one obtains the optimal $L^2$-estimate, but it needs higher regularity, that is, $u_0\in H^4(\Omega) \cap H^1_0(\Omega)$ and $u_1\in H^3(\Omega) \cap H^1_0(\Omega).$ However, in the following theorem, using a nonstandard method by Baker \cite{Baker}, \cite{PY} (see also \cite{larsson1991finite} for a different approach), we derive  the optimal $L^2$-error estimate with minimal regularity in the initial data, that is, $u_0\in H^3(\Omega) \cap H^1_0(\Omega)$ and $u_1\in H^2(\Omega) \cap H^1_0(\Omega).$
\begin{thm}\label{thm:L2est}
	Let $u$ and $u_h$ be the solution of \eqref{se1.1} and \eqref{sdf}, respectively. Then, there exists a positive constant $C$ independent of $h$ such that
	\[
	\|u(t) - u_h(t)\| \leq Ch^2 t^\half e^{-\delta t}\left(\|u_0\|_3 + \|u_1\|_2\right).
	\]
\end{thm}

\begin{proof}
	On integrating equation \eqref{eq:error} with respect to time, we now obtain 
	\begin{equation}\label{eq:theta-L2}
		\begin{aligned}
			\left(\theta', v\right) + \beta~a_h(\theta,v) + \alpha\left(\theta,v\right) + a_h(\hat{\theta},v) &= -\left(\eta',v\right) - \alpha\left(\eta,v\right) 
			\\
			&+ \alpha\left(e(0), v\right) + \left(e'(0),v\right) + \beta~a_h(\theta(0),v).
		\end{aligned}
	\end{equation}
	Choose  $u_{0h}$ and $u_{1h}$ as elliptic and $L^2$ projections of $u_0$ and $u_1$, respectively, then,
		\[
		\left(e'(0), v\right) = 0 ~ \mbox{and}~ \theta(0) = 0.
		\]
	Set $v = \theta + \tilde{\delta} \hat{\theta}$ in \eqref{eq:theta-L2}  and then use definiton of the energy and Young's inequality to obtain
	\begin{equation}\label{eq:sec4err}
		\begin{aligned}
			\frac{{\rm d}}{{\rm d}t}\tilde{E}_{\tilde{\delta}}(\theta)(t) + \mathcal{F}(\theta)(t) \leq C\left(\|\eta'\|^2 + \|\eta\|^2 + \|\eta(0)\|^2\right), %+ \delta_0\left(\|\theta\|^2 + \lambda_1\|\hat{\theta}\|^2\right),
		\end{aligned}
	\end{equation}
	where
	\begin{equation*}
		\tilde{E}_{\tilde{\delta}}(\theta)(t) = E(\theta)(t) + \tilde{\delta}\left(\theta, \hat{\theta}\right) + \frac{\beta\tilde{\delta}}{2}|\hat{\theta}|^2_1 + \frac{\alpha\tilde{\delta}}{2}\|\hat{\theta}\|^2,
	\end{equation*}
	with
	\begin{equation*}
		E(\theta)(t) = \half\left(\|\theta\|^2 + |\hat{\theta}|^2_1\right),
	\end{equation*}
	and 
	\begin{equation*}
		\mathcal{F}(\theta)(t) = \left(\tilde{\delta}-\delta_0\right)|\hat{\theta}|^2_1 + \left(\alpha+\beta\lambda_1-\delta_0-\tilde{\delta}\right)\|\theta\|^2.
	\end{equation*}
	A use of Young's inequality, provided $0 < {\tilde{\delta}} \leq \min\left(\frac{\left(\alpha+\beta\lambda_1\right)}{2}, \frac{\lambda_1}{4\left(\alpha+\beta\lambda_1\right)}\right)$, yields
	\begin{equation*}
		\begin{aligned}
			\tilde{E}_{\tilde{\delta}}(\theta)(t) &\leq E(\theta)(t) + \frac{{\tilde{\delta}}}{2\left(\alpha+\beta\lambda_1\right)}\|\theta\|^2 + \frac{\left(\alpha+\beta\lambda_1\right){\tilde{\delta}}}{2}\|\hat{\theta}\|^2 + \frac{\left(\alpha+\beta\lambda_1\right){\tilde{\delta}}}{2\lambda_1}|\hat{\theta}|^2_1
			\\
			&\leq \frac{3}{2}E(\theta)(t) + \left(\frac{{\tilde{\delta}}}{2\left(\alpha+\beta\lambda_1\right)} - \frac{1}{4}\right)\|\theta\|^2 + \left(\frac{\left(\alpha+\beta\lambda_1\right){\tilde{\delta}}}{\lambda_1} - \frac{1}{4}\right)|\hat{\theta}|^2_1
			\\
			&\leq \frac{3}{2}E(\theta)(t).
		\end{aligned}
	\end{equation*}
	Similarly, we obtain
	\begin{equation*}
		\tilde{E}_{\tilde{\delta}}(\theta)(t) \geq \half E(\theta)(t),
	\end{equation*}
	provided $0 < {\tilde{\delta}}\leq \min\left(\frac{\left(\alpha+\beta\lambda_1\right)}{2}, \frac{\lambda_1}{4\left(\alpha+\beta\lambda_1\right)}\right)$. Hence,
	\begin{equation}\label{eq:sec4err1}
	\half E(\theta)(t) \leq	\tilde{E}_{\tilde{\delta}}(\theta)(t) \leq \frac{3}{2} E(\theta)(t),
	\end{equation}
	provided $0 < {\tilde{\delta}}\leq \min\left(\frac{\left(\alpha+\beta\lambda_1\right)}{2}, \frac{\lambda_1}{4\left(\alpha+\beta\lambda_1\right)}\right)$.
	From the definition of $\mathcal{F}(\theta)$, provided $0 < {\tilde{\delta}} \leq \frac{\alpha+\beta\lambda_1}{2}$, we obtain
	\begin{equation}\label{eq:sec4err2}
		\begin{aligned}
			\mathcal{F}(\theta)(t) &= \left({\tilde{\delta}}-\delta_0\right)|\hat{\theta}|^2_1 + \left(\alpha+\beta\lambda_1-\delta_0-\tilde{\delta}\right)\|\theta\|^2
			\\
			&\geq 2\left(\tilde{\delta}-\delta_0\right) E(\theta)(t) + \left(\alpha+\beta\lambda_1 - 2\tilde{\delta}\right)\|\theta\|^2
			\\
			&\geq 2\left(\tilde{\delta}-\delta_0\right) E(\theta)(t) \geq \frac{4}{3}\left(\tilde{\delta} - \delta_0\right)E_{\tilde{\delta}}(\theta)(t).
		\end{aligned}
	\end{equation}
	A use of the above two inequality \eqref{eq:sec4err1}-\eqref{eq:sec4err2} into \eqref{eq:sec4err} and an integration in time yields
	\begin{equation}\label{eq:sec4err3}
		\begin{aligned}
			\tilde{E}_{\tilde{\delta}}(\theta)(t) &\leq e^{-\frac{4}{3}\left(\tilde{\delta}-\delta_0\right)t}\tilde{E}_{\tilde{\delta}}(\theta)(t) + C\int_0^t e^{-\frac{4}{3}\left(\tilde{\delta}-\delta_0\right)\left(t-s\right)}\left(\|\eta'(s)\|^2+\|\eta(s)\|^2+\|\eta(0)\|^2\right)\,{\rm d}s
			\\
			&\leq \frac{3}{2}e^{-2\delta t}E(\theta)(0) + C\int_0^t e^{-2\delta\left(t-s\right)}\left(\|\eta'(s)\|^2+\|\eta(s)\|^2+\|\eta(0)\|^2\right)\,{\rm d}s,
		\end{aligned}
	\end{equation}
	where $\delta = \frac{2}{3}\left(\tilde{\delta}-\delta_0\right)$.
	Since $E(\theta)(0) = \half\theta(0) = 0$. A use of equation \eqref{eq:2.41} yields
	\begin{equation*}
		\begin{aligned}
			\|\eta'(s)\|^2 + \|\eta(s)\|^2 \leq Ch^4\left(\|u'(s)\|^2_2 + \|u(s)\|^2_2\right) 
			\leq Ch^4e^{-2\delta s}\left(\|u_1\|^2_2 + \|u_0\|^2_3\right).
		\end{aligned}
	\end{equation*}
	Using the above equation from equation \eqref{eq:sec4err3}, we obtain
	\begin{equation*}
		\|\theta(t)\|^2 \leq Ch^4te^{-2\delta t}\left(\|u_1\|^2_2 + \|u_0\|^2_3\right).
	\end{equation*}
	This completes the proof.
	\end{proof}

When $d = 2,$ the following theorem on the maximum norm estimate is obtained from the  super-convergence result  \eqref{estimate:super-cgt}.

\begin{thm}\label{thm:max-norm}
Let $u_{0h} = \Pi_hu_0$ 
%then $\theta(0) = 0$ 
and let $u_{1h}$ either the $L^2$-projection or the interpolation of $u_1$ in $X_h$. Then, there holds for $d = 2$, 
\begin{equation*}\label{estimate:max-norm}
	\|u(t) - u_h(t)\|_{L^\infty} \leq Ch^2\left(\log\left(\frac{1}{h}\right)\right)e^{-\delta t}\left(\|u_1\|_4+\|u_0\|_5\right).
\end{equation*}
\end{thm}
\begin{proof}
For $d = 2$ and from the Sobolev inequality \cite[Lemma 6.4, p. 88]{thomee2007galerkin}, we obtain
\begin{equation*}
	\|\theta(t)\|_{L^\infty} \leq C\left(\log\left(\frac{1}{h}\right)\right)^\half\|\nabla\theta\| \leq C\left(\log\left(\frac{1}{h}\right)\right)^\half h^2(1+t)^\half e^{-2\delta t}\left({\xb\beta\|u_1\|_4^2} + \|u_1\|_3^2 + \|u_0\|_4^2\right).
\end{equation*}
Since 
\begin{equation*}
	\|\eta(t)\|_{L^\infty} \leq Ch^2\left(\log\left(\frac{1}{h}\right)\right)\|u(t)\|_{W^{2,\infty}} \leq Ch^2\left(\log\left(\frac{1}{h}\right)\right)e^{-\delta t}\left(\|u_1\|_4+\|u_0\|_5\right),
\end{equation*}
then, for $d=2,$ a use of triangle inequality completes the proof.
\end{proof}

%%%%%%%%%%%%%%%%%%%%%%%%%%%%%%%%%%%%%%%%%%%%%%%%%%%%%%%%%%%%%%%%%%%%%%%%%%
%%%%%%%%%%%%%%%%%%%%%%%%%%%%%%%%%%%%%%%%%%%%%%%%%%%%%%%%%%%%%%%%%%%%%%%%
\section{Completely Discrete Scheme}\label{sec:fulld}
%%%%%%%%%%%%%%%%%%%%%%%%%%%%%%%%%%%%%%%%%%%%%%%%%%%%%%%%%%%%%%%%%%%%%%%%%%
%%%%%%%%%%%%%%%%%%%%%%%%%%%%%%%%%%%%%%%%%%%%%%%%%%%%%%%%%%%%%%%%%%%%%%%%
This section focuses on a fully discrete scheme, which is based on a finite difference discretization in time for the semidiscrete problem  and it is shown that uniform exponential decay property is preserved for the discrete solution.

Let $k > 0$ be the time step and let $t_n = n k, \; n \geq 0$. Set 
$\varphi^n = \varphi(t_n)$,
$$
{\bar \partial}_t \varphi^n = \frac{\varphi^n - \varphi^{n-1}}{k} \;\; \mbox{and} \;\; \partial_t \varphi^n = \frac{\varphi^{n + 1} - \varphi^{n}}{k} \;\; \mbox{with} \;\; {\bar \partial}^0_t \varphi^n = \varphi^n. 
$$
%with ${\bar \partial}^0_t \varphi^n = \varphi^n$. \\
Define
$$
{\bar \partial}^{(j + 1)}_t \varphi^n = \frac{1}{k} \left( {\bar \partial}^{j}_t \varphi^n - 
{\bar \partial}^{j}_t \varphi^{n - 1} \right), \;\; j \geq 0.
$$
Let $\varphi^{n + \frac{1}{2}} = \frac{\varphi^{n + 1} + \varphi^n}{2}$ and set 
\begin{eqnarray*}
	%\delta_t \varphi^n &=&  \frac{\varphi^{n+1} - \varphi^{n-1}}{2 k} = {\bar \partial}_t \varphi^{n + \frac{1}{2}} = \frac{\varphi^{n + \frac{1}{2}} - \varphi^{n - \frac{1}{2}}}{k}, \\
	%{\hat \varphi}^n &=& \frac{1}{4} \left(\varphi^{n+1} + 2 \varphi^n + \varphi^{n-1} \right) = \frac{1}{2} \left(\varphi^{n + \frac{1}{2}} + \varphi^{n - \frac{1}{2}} \right), \\
	\partial_t {\bar \partial}_t \varphi^n &=& \frac{1}{k^2} \left(\varphi^{n + 1} - 2 \varphi^n + \varphi^{n -1} \right) = \frac{1}{k} \left(\partial_t \varphi^n - {\bar \partial}_t \varphi^n \right).
\end{eqnarray*}
The discrete time finite element approximations $U^n$ of
$u(t_n)$ are defined as solution of 
\begin{eqnarray} %\label{fullweak-1}
	({\bar \partial}_t\partial_t U^n, \chi) + \beta~a_h(\partial_t U^n, \chi) + \alpha (\partial_t U^n, \chi) 
	+ a_h(U^{n+1}, \chi) = 0, \; \chi \in X_h, \; n \geq 1, \label{Feqn3.1}
\end{eqnarray}
with $U^0 = u_{0,h}$ and $U^1 = u_{1,h},$ where $u_{0,h}, u_{1,h} \in X_h$ are appropriate approximations to be
defined later.

%%%%%%%%%%%%%%%%%%%%%%%%%%%%%%%%%%%%%%%%%%%%%%%%%%%%%%%%%%%%%%%%%%%%%%%%
%%%%%%%%%%%%%%%%%%%%%%%%%%%%%%%%%%%%%%%%%%%%%%%%%%%%%%%%%%%%%%%%%%%%%%%%

The next lemma proves some properties of the discrete energy
%\begin{enumerate}[label=(\roman*)]
 $$E^n(U) = \half\left(\|\partial_tU^n\|^2 + |U^{n+1}|^2_1\right).$$
\begin{Lem}\label{Lem:fulene}
	The solution $U^{n+1}$ for $n \geq 1$ of the fully discrete scheme \eqref{Feqn3.1} satisfies for $0\leq \vartheta \leq 1$
		\begin{equation}\label{eq:Fdereg}
			\begin{aligned}
				\bar{\partial}_tE^n(U) &\leq - \vartheta\left(\alpha+\beta\lambda_1\right)\|\partial_tU^n\|^2 - \left(1-\vartheta\right)\alpha\|\partial_tU^n\|^2 - \left(1 - \vartheta\right)\beta|\partial_tU^n|^2_1 
				\\
				&\qquad- \frac{k}{2}\left(\|\bar{\partial}_t\partial_tU^n\|^2 + |\partial_tU^n|^2_1\right).
			\end{aligned}
		\end{equation}
		Moreover, for the extended energy 
		\begin{equation}\label{def:disexen}
			\tilde{E}^n_\delta(U) = E^n(U) + \delta\left(\partial_tU^n, U^{n+1}\right),
		\end{equation}
        there holds
		\begin{equation}\label{eq:Fengrel}
			\frac{1}{2}E^n(U) \leq \tilde{E}^n_\delta(U) \leq \frac{3}{2}E^n(U), 
		\end{equation}
		provided $ 0 < \delta \leq \min\left(\frac{\alpha+\beta\lambda_1}{2}, \frac{\lambda_1}{2\left(\alpha+\beta\lambda_1\right)}\right)$.
	%\end{enumerate}
\end{Lem}
\begin{proof}
   Choose $\chi = \partial_tU^n$ in \eqref{Feqn3.1}, a use of  the  \eqref{dispoincare} with $\vartheta \in (0,1)$, the Cauchy-Schwartz inequality and the Young's inequality shows
   
   \begin{equation*}
   \begin{aligned}
   		\bar{\partial}_tE^n(U) + \vartheta\left(\alpha+\beta\lambda_1\right)\|\partial_tU^n\|^2 &+ \left(1-\vartheta\right)\beta|\partial_tU^n|^2_1 + \left(1-\vartheta\right)\alpha\|\partial_tU^n\|^2 
   		\\
   		&\hspace*{2cm}+ \frac{k}{2}\left(\|\bar{\partial}_t\partial_tU^n\|^2 + |\partial_tU^n|^2_1\right) \leq 0.
   	\end{aligned}
   \end{equation*}
   This completes the proof of \eqref{eq:Fdereg}. For the proof of \eqref{eq:Fengrel}, an application  of the H\"older inequality with  \eqref{dispoincare} and the Young's inequality yields
   \begin{equation}\label{eq:extrineq}
   	\begin{aligned}
   		\Big\vert\left(\partial_tU^n, U^{n+1}\right)\Big\vert \leq \|\partial_tU^{n+1}\|\|U^{n+1}\| \leq \frac{1}{\sqrt{\lambda_1}}\|\partial_tU^n\|\,|U^{n+1}|_1 
   		\\
   		\leq \frac{1}{2\left(\alpha+\beta\lambda_1\right)}\|\partial_tU^n\|^2 + \frac{\left(\alpha+\beta\lambda_1\right)}{2\lambda_1}|U^{n+1}|^2_1.
   	\end{aligned}
   \end{equation} 
   Using \eqref{eq:extrineq} in the definition of the extended energy \eqref{def:disexen}, we arrive at
   \begin{equation*}
   	\begin{aligned}
   		\tilde{E}_\delta^n(U) &= E^n(U) + \delta\left(\partial_tU^n, U^{n+1}\right) 
   		\\
   		&\geq \frac{1}{2}E^n(U) + \left(\frac{1}{4} - \frac{1}{2\left(\alpha+\beta\lambda_1\right)}\delta\right)\|\partial_tU^n\|^2 + \left(\frac{1}{4} - \frac{\left(\alpha+\beta\lambda_1\right)}{2\lambda_1}\delta\right)|U^{n+1}|_1^2
   		\\
   		&\geq \frac{1}{2}E^n(U), \;\; \mbox{provided}\; 0 < \delta \leq \min\left(\frac{\alpha+\beta\lambda_1}{2},\frac{\lambda_1}{2\left(\alpha+\beta\lambda_1\right)}\right).
   	\end{aligned}
   \end{equation*}
   Again from the definition of the extended energy \eqref{def:disexen}, a use of \eqref{eq:extrineq} shows
   \begin{equation*}
   	\begin{aligned}
   		\tilde{E}_\delta^n(U) &= E^n(U) + \delta\left(\partial_tU^n, U^{n+1}\right) 
   		\\
   		&\leq \frac{3}{2}E^n(U) - \left(\frac{1}{4} - \frac{1}{2\left(\alpha+\beta\lambda_1\right)}\delta\right)\|\partial_tU^n\|^2 - \left(\frac{1}{4} - \frac{\left(\alpha+\beta\lambda_1\right)}{2\lambda_1}\delta\right)|U^{n+1}|_1^2
   		\\
   		&\leq \frac{3}{2}E^n(U), \;\; \mbox{provided}\; 0 < \delta \leq \min\left(\frac{\alpha+\beta\lambda_1}{2},\frac{\lambda_1}{2\left(\alpha+\beta\lambda_1\right)}\right).
   	\end{aligned}
   \end{equation*}
   This completes the rest of the proof.
\end{proof}

%%%%%%%%%%%%%%%%%%%%%%%%%%%%%%%%%%%%%%%%%%%%%%%%%%%%%%%%%%%%%%%%%%%%%%%%%%%%%%%%%%%%%%%%%%
%%%%%%%%%%%%%%%%%%%%%%%%%%%%%%%%%%%%%%%%%%%%%%%%%%%%%%%%%%%%%%%%%%%%%%%%%%%%%%%%%%%%%%%%%%
The following theorem deals with uniform exponential decay property of the completely discrete energy.
\begin{thm}\label{thm:fulene}
	For $0 < \delta \leq \min\left(\frac{\alpha + \beta\lambda_1}{2}, \frac{\lambda_1}{2\left(\alpha+\beta\lambda_1\right)}\right)$, the solution $U^{n+1}$ for $n \geq 1$ of \eqref{Feqn3.1} satisfies
	\begin{equation*}
		E^n(U) \leq 3e^{-\frac{1}{15}\delta t_n}E^0(U),
	\end{equation*}
	where $E^n(U) = \half\left(\|\partial_tU^n\|^2 + |U^{n+1}|^2_1\right)$. 
\end{thm}

\begin{proof}
	From the definition of the extended energy \eqref{def:disexen}, there holds
	\begin{equation}\label{eq:Fderexen}
		\begin{aligned}
			\bar{\partial}_t\tilde{E}^n_\delta(U) &= \bar{\partial}_tE^n(U) + \delta\bar{\partial}_t\left(\partial_tU^n, U^{n+1}\right)
			\\
			&= \bar{\partial}_tE^n(U) + \frac{\delta}{k}\left(\left(\partial_tU^n, U^{n+1}\right)-\left(\partial_tU^n, U^n\right)+\left(\partial_tU^n, U^n\right)-\left(\partial_tU^{n-1}, U^n\right)\right)
			\\
			&= \bar{\partial}_tE^n(U) + \delta\|\partial_tU^n\|^2 + \delta\left(\bar{\partial}_t\partial_tU^n, U^n\right).
		\end{aligned}
	\end{equation}
	Setting  $\chi = U^n$  in \eqref{eq:Fderexen} with the relation $U^n = U^{n+1} - k\partial_tU^n$, a  use of the H\"older inequality with  \eqref{dispoincare} and  the Young's inequality implies
	\begin{equation*}
		\begin{aligned}
			\left(\bar{\partial}_t\partial_tU^n, U^n\right) &= -\beta a_h(\partial_tU^n, U^n) - \alpha\left(\partial_tU^n, U^n\right) - a_h(U^{n+1}, U^n)
			\\
			& = -\beta a_h(\partial_tU^n,U^{n+1}) + \beta k|\partial_t U^n|^2_1 - \alpha\left(\partial_tU^n,U^{n+1}\right) + \alpha k\|\partial_tU^n\|^2 
			\\
			&\qquad- |U^{n+1}|^2_1 + ka_h(U^{n+1}, \partial_tU^n)
			\\
			&\leq \frac{\beta^2}{2\delta_0}|\partial_tU^n|^2_1 + \frac{\delta_0}{2}|U^{n+1}|^2_1 + \left(\frac{\beta^2}{2\delta_1} + \frac{k^2\delta_1}{2}\right)|\partial_tU^n|^2_1 + \frac{\alpha^2}{2\lambda_1\delta_0}\|\partial_tU^n\|^2 
			\\
			&\qquad + \frac{\delta_0}{2}|U^{n+1}|^2_1 + \left(\frac{\alpha^2}{2\lambda_1\delta_1} + \frac{k^2\lambda_1\delta_1}{2}\right)\|\partial_tU^n\|^2 - |U^{n+1}|^2_1 
			\\
			&\quad\qquad + \frac{\delta_2}{2}|U^{n+1}|^2_1 + \frac{k^2}{2\delta_2}|\partial_tU^n|^2_1\\
			%\end{aligned}
	%\end{equation*}
   %\begin{equation*}
		%\begin{aligned}         
			&\quad\leq -\left(1 - \delta_0 - \frac{\delta_2}{2}\right)|U^{n+1}|^2_1 + \frac{\beta\left(\alpha+\beta\lambda_1\right)}{2\lambda_1}\left(\frac{1}{\delta_0} + \frac{1}{\delta_1}\right)|\partial_tU^n|^2_1
			\\
			& \qquad + \frac{\alpha\left(\alpha+\beta\lambda_1\right)}{2\lambda_1}\left(\frac{1}{\delta_0} + \frac{1}{\delta_1}\right)\|\partial_tU^n\|^2 + k^2\left(\delta_1+\frac{1}{2\delta_2}\right)|\partial_tU^n|^2_1.
		\end{aligned}
	\end{equation*}
	A use of the above inequality with \eqref{eq:Fdereg} with a choice of $\vartheta = \frac{21}{40}, \delta_0 = \frac{3}{4}, \delta_2 = \frac{2}{5}, \delta_1 = \frac{30}{17}$ in \eqref{eq:Fderexen} provided $0 < \delta \leq \min\left(\frac{\alpha+\beta\lambda_1}{2},\frac{\lambda_1}{2\left(\alpha+\beta\lambda_1\right)}\right)$ and $\delta k \leq \frac{34}{205}$, shows
	\begin{equation*}
		\begin{aligned}
			\bar{\partial}_t\tilde{E}^n_\delta(U) &\leq -\frac{\delta}{20}|U^{n+1}|^2_1 - \left(\frac{21}{40}\left(\alpha+\beta\lambda_1\right) - \delta\right)\|\partial_tU^n\|^2 - \beta\left(\frac{19}{40} - \frac{19}{20}\frac{\left(\alpha+\beta\lambda_1\right)}{\lambda_1}\delta\right)\,|\partial_tU^n|^2_1
			\\
			&\qquad - \alpha\left(\frac{19}{40} - \frac{19}{20}\frac{\left(\alpha+\beta\lambda_1\right)}{\lambda_1}\delta\right)\,\|\partial_tU^n\|^2 - \frac{k}{2}\|\bar{\partial}_t\partial_tU^n\|^2 - \frac{k}{2}\left(1 - \frac{205}{34}\delta k\right)\,|\partial_tU^n|^2_1
			\\
			&\leq -\frac{1}{10}\delta E^n(U) - \left(\frac{21}{40}\left(\alpha+\beta\lambda_1\right) - \frac{21}{20}\delta\right)\|\partial_tU^n\|^2 - \frac{k}{2}\left(1 - \frac{205}{34}\delta k\right)\,|\partial_tU^n|^2_1
			\\
			&\leq -\frac{1}{10}\delta E^n(U).
		\end{aligned}
	\end{equation*}
	 Now, proceed iteratively to arrive at
	\begin{equation*}
		\begin{aligned}
			\tilde{E}^n_\delta(U) \leq \frac{1}{1+\frac{1}{15}\delta k}\tilde{E}^{n-1}_\delta &\leq e^{-\frac{1}{15}\delta k}\tilde{E}^{n-1}_\delta(U) \leq \dots \leq e^{-\frac{1}{15}\delta t_n}\tilde{E}^0_\delta(U),
		\end{aligned}
	\end{equation*}
    and  from \eqref{eq:Fengrel}
	\begin{equation*}
    \begin{aligned}
			E^n(U) \leq 2\tilde{E}_\delta^n(U) &\leq 2e^{-\frac{1}{15}\delta t_n}\tilde{E}^0_\delta(U) \leq 3e^{-\frac{1}{15}\delta t_n} E^0(U).
		\end{aligned}
	\end{equation*}
	This completes the rest of the proof.
\end{proof}

%%%%%%%%%%%%%%%%%%%%%%%%%%%%%%%%%%%%%%%%%%%%%%%%%%%%%%%%%%%%%%%%%%%%%%%%
%%%%%%%%%%%%%%%%%%%%%%%%%%%%%%%%%%%%%%%%%%%%%%%%%%%%%%%%%%%%%%%%%%%%%%%%
\section{Some Generalizations}\label{sec:gen}
%%%%%%%%%%%%%%%%%%%%%%%%%%%%%%%%%%%%%%%%%%%%%%%%%%%%%%%%%%%%%%%%%%%%%%%%
%%%%%%%%%%%%%%%%%%%%%%%%%%%%%%%%%%%%%%%%%%%%%%%%%%%%%%%%%%%%%%%%%%%%%%%%%%%%%
This section focuses on some generalizations such as problems with nonhomogeneous right-hand sides, with space and time varying damping coefficients, and beam equations with a damping term. Finally, an abstract semi-discrete problem is formulated with some assumptions, which provides a uniform decay property, and, as a consequence, finite difference as well as spectral methods are briefly discussed, whose solutions show uniform asymptotic behavior.   

%%%%%%%%%%%%%%%%%%%%%%%%%%%%%%%%%%%%%%%%%%%%%%%%%%%%%%%%%%%%%%%%%%%%%%%%%%
\subsection{Inhomogeneous Equations} 
%%%%%%%%%%%%%%%%%%%%%%%%%%%%%%%%%%%%%%%%%%%%%%%%%%%%%%%%%%%%%%%%%%%%%%%%%%
This subsection discusses the strongly damped wave equation with time independent source term such as 
\begin{equation*}
	u'' + \beta Au' + \alpha u' + Au = f \quad \mbox{in}\quad \Omega\times (0,\infty)
\end{equation*}
with initial conditions
\[
 u(0) = u_0, \quad u'(0) = u_1.
\]
 Let $u_\infty$ be the unique solution of the steady-state equation
		\[
		 Au_\infty = f \quad \mbox{with}\quad u_\infty = 0 \quad \mbox{on} \quad \partial\Omega.
		\]
Then, we have the following theorem.     
\begin{thm}\label{thm:in-eng}
	Let  $w = u(t) - u_\infty.$ Then, there holds  
	\begin{enumerate}[label=(\roman*)]
		\item for $0<\delta < \min\left(\frac{\alpha+\beta\lambda_1}{2},\frac{\lambda_1}{\alpha+\beta\lambda_1}\right)$ or \, $0 < \delta = \frac{\alpha+\beta\lambda_1}{2},$
		\begin{equation*}
			E^{(k)}_{A^{(j)}}(w)(t) \leq C(\alpha,\beta,\lambda_1)e^{-2\delta t}E^{(k)}_{A^{(j)}}(w)(0) = C(\alpha,\beta,\lambda_1)e^{-2\delta t}\left(|w^{(k)}(0)|^2_{2j-1} + |w^{(k-1)}(0)|^2_{2j}\right).
		\end{equation*}
		\item For $0<\delta = \frac{\lambda_1}{\alpha+\beta\lambda_1},$ 
		\begin{equation*}
			|w^{(k)}(t)|^2_{2j-1} \leq C(\alpha,\beta,\lambda_1)e^{-2\delta t}E^{(k)}_{A^{(j)}}(w)(0) = C(\alpha,\beta,\lambda_1)e^{-2\delta t}\left(|w^{(k)}(0)|^2_{2j-1} + |w^{(k-1)}(0)|^2_{2j}\right).
		\end{equation*}
	\end{enumerate}
\end{thm}

\begin{proof}
Note that $w(t)$ satisfies
	\begin{equation}
		\begin{aligned} \label{eq-w}
			w'' + \beta Aw' + \alpha w' + Aw = 0 \quad \mbox{in} \quad \Omega\times (0,\infty),
			\\ 
			w(0) = u_0 - u_\infty, \quad w'(0) = u_1.
            %\label{eq-w-initial}
		\end{aligned}
	\end{equation}
Therefore, for problem \eqref{eq-w}, a use of techniques similar to the proof of decay properties in Theorems \ref{thm:eng}-\ref{thm:Aeng}, replacing $u$ by $w$ completes the rest of the proof.
\end{proof}
\begin{Rem}
	For $d = 2,3$, a use of the Sobolev embedding result shows
	\begin{equation*}
		\|w(t)\|_{L^\infty} \leq C\|w(t)\|_2 \leq Ce^{-\delta t}\left(\|u_1\|_3 + \|u_0 - u_\infty\|_3\right).
	\end{equation*}
	This implies $u(t) \to u_\infty$ as $t\to\infty$ with exponential decay rate $\delta.$
\end{Rem}
For the semidiscrete solution $w_h(t) = u_h(t) - u_{\infty,h}$ similar results as in Section \ref{sec:semd} hold true and therefore, $\|u_h(t) - u_{\infty,h}\|_{L^\infty} = O(e^{-\delta t})$.
\begin{Rem}
	If $f(t) = O(e^{-\delta_0 t})$, then the solution also decays exponentially with the decay rate $\delta^* = \min(\delta,\delta_0).$
\end{Rem}
%%%%%%%%%%%%%%%%%%%%%%%%%%%%%%%%%%%%%%%%%%%%%%%%%%%%%%%%%%%%%%%%%%%%%%%%%%%%%%%
\subsection{Time varying damping parameters} \label{subse-time-varying}
%%%%%%%%%%%%%%%%%%%%%%%%%%%%%%%%%%%%%%%%%%%%%%%%%%%%%%%%%%%%%%%%%%%%%%%%%%%%%%%
For equation \eqref{se1.1}, assume that $\alpha =\alpha(t)$ and $\beta= \beta(t)$ and then we derive, below, in terms of a theorem, the exponential decay behavior of the solution.
\begin{thm}\label{thm:a(t):eng}
	Let $\alpha = \alpha(t)$ and $\beta = \beta(t)$ in equation \eqref{se1.1}. Then, under the assumption that
	\[
	0 < \alpha_{\min} \leq \alpha(t) \leq \alpha_{\max}, \; 0 < \beta_{\min} \leq \beta(t) \leq \beta_{\max} \; \mbox{and} \; \alpha'(t) \geq 0,~ \beta'(t) \geq 0, \; \forall~ t \in [0,\infty),
	\]
	the following results hold for $u_0 \in D(A^{(j)})$ and $u_1 \in D(A^{(j-\half)})$ with $j\geq 1/2$ 
	\begin{enumerate}[label=(\roman*)]
		%\begin{enumerate}[label=(\roman*)]
    	\item if $0 < \delta < \min\left(\frac{\alpha_{\min} + \beta_{\min}\lambda_1}{2}, \frac{\lambda_1}{\alpha_{\max}+\beta_{\max}\lambda_1}\right)$ or \, $0 < \delta = \frac{\alpha_{\min} + \beta_{\min}\lambda_1}{2}$, then
    	\[
    	E_{A^{(j)}}(u)(t) \leq C(\alpha,\beta,\lambda_1)e^{-2\delta t}E_{A^{(j)}}(u)(0),
    	\]         
    	where $E_{A^{(j)}}(u)(t) = \half\left(|u'(t)|^2_{2j-1} + |u(t)|^2_{2j}\right)$.
    	\item If $0 < \delta = \frac{\lambda_1}{\alpha_{\max} + \beta_{\max}\lambda_1}$, then
    	\[
    	|u'(t)|^2_{2j-1} \leq C e^{-2\delta t}E_{A^{(j)}}(u)(0).
    	\]
    \end{enumerate}	 
\end{thm}
\begin{proof}
It is enough to prove the result for $j=1/2,$ that is,
\begin{equation*}
		E(u)(t) \leq Ce^{-2\delta t}E(u)(0),
\end{equation*}
where $E(u)(t) = \half\left(\|u'\|^2 + |u|^2_1\right).$ Then, the rest of the analysis follows from the proof technique of Theorems~\ref{thm:Aeng} and \ref{thm:geno} for $k=1$. If $\alpha$ and $\beta$ depends on the time $t$, then equation \eqref{se11.4} will change into
	\begin{equation}\label{se11.4t}
		\begin{aligned}
			(u''_\delta, v) + \beta(t)~a(u'_\delta, v) + (\alpha(t) - 2\delta)(u'_\delta, v) &+ (1 - \beta(t)\delta)~a(u_\delta, v)
			\\
			&- \delta(\alpha(t) - \delta)(u_\delta,v) = 0,
		\end{aligned}
	\end{equation}
	with
	\begin{eqnarray*}
		u_\delta(x,0) = u_0 \; \mbox{and} \; (u_\delta)'(x,0) = \delta u_0 + u_1 \; \mbox{in} \; \Omega.
	\end{eqnarray*}
	By setting $v = u_\delta'$ and using \eqref{poincare} in the  equation \eqref{se11.4t}, we obtain
	\begin{equation}\label{t-in}
		\frac{{\rm d}}{{\rm d}t}\mathcal{\tilde{E}}(u_\delta)(t) + \left(\alpha(t) + \beta(t)\lambda_1 - 2\delta\right)\|u'_\delta\|^2 + \frac{\delta}{2}\beta'(t)|u_\delta|^2_1 + \frac{\delta}{2}\alpha'(t)\|u_\delta\|^2 \leq 0,
	\end{equation}
	where
	\begin{equation*}
		\mathcal{\tilde{E}}(u_\delta)(t) = \half\left(\|u'_\delta\|^2 + \left(1 - \beta(t)\delta\right)|u_\delta|^2_1 - \delta\left(\alpha(t) - \delta\right)\|u_\delta\|^2\right).
	\end{equation*}
	An integration of \eqref{t-in} with respect to time yields
    for $\delta \leq \frac{\alpha_{\min} + \beta_{\min}\lambda_1 }{2}$ and $\beta'(t) \geq 0, \alpha'(t) \geq 0$
	\begin{equation} \label{se5.63}
		\begin{aligned}
			\mathcal{\tilde{E}}(u_\delta)(t) \leq \mathcal{\tilde{E}}(u_\delta)(0) \leq C E(u)(0). %= \half\left(\|u_1 - \delta u_0\|^2 + \left(1 - \beta(0)\delta\right)\|A^\half u_0\|^2 - \delta(\alpha(0) - \delta)\|u_0\|^2\right)
		\end{aligned}
	\end{equation} 
	Now, a use of \eqref{poincare} shows
	\begin{equation*}\label{eq:matet}
		\begin{aligned}
			\mathcal{\tilde{E}}(u_\delta)(t) \geq \half\left(\|u'_\delta\|^2 + \left(1 - \frac{\alpha_{\max} + \beta_{\max}\lambda_1}{\lambda_1}\delta\right)|u_\delta|^2_1 + \delta^2\|u_\delta\|^2\right)	
			> \gamma_0 E(u_\delta)(t),
		\end{aligned}
	\end{equation*}  
	provided $\delta < \gamma_0$, where $\gamma_0 = \frac{\lambda_1}{\alpha_{\max} + \beta_{\max}\lambda_1}$. Now, when $\delta = \gamma_0$, then there holds
	\begin{equation*}
		\mathcal{\tilde{E}}(u_\delta)(t) \geq \half\left(\|u'_\delta\|^2 + \delta^2\|u_\delta\|^2\right) \geq e^{2\delta t}\|u'\|^2. 
	\end{equation*}
	This completes the rest of the proof.
\end{proof}

%%%%%%%%%%%%%%%%%%%%%%%%%%%%%%%%%%%%%%%%%%%%%%%%%%%%%%%%%%%%%%%%%%%%%%%%%%%%%%%%%
Below, we discuss the decay estimates of higher order in time energy. 
\begin{thm}\label{thm:a(t):1steng}
Under the hypothesis of  Theorem ~\ref{thm:a(t):eng} and $\max_{t\in [0,\infty)} (|\beta''(t)|, |\alpha''(t)|)\leq M,$  there hold
	\begin{enumerate}[label=(\roman*)]
		\item if $\delta < \min\left(\frac{\alpha_{\min} + \beta_{\min}\lambda_1}{2}, \frac{\lambda_1}{\alpha_{\max} + \beta_{\max}\lambda_1}\right)$ or ~ $\delta = \frac{\alpha_{\min} + \beta_{\min}\lambda_1}{2}$, then
		\begin{equation*}
			E(u')(t) \leq C(1+t)\;e^{-2\delta t}\;( E(u')(0) + E_{A^{1/2}}(u)(0)),
		\end{equation*}
		where $E(u')(t) = \half\left(\|u''\|^2 + |u'|^2_1\right)$.
		\item If $\delta = \frac{\lambda_1}{\alpha_{\max} + \beta_{\max}\lambda_1}$, then
		\begin{equation*}
			\|u''\|^2 \leq C  (1+t)\;e^{-2\delta t} ( E(u')(0) + E_{A^{1/2}}(u)(0)).
		\end{equation*}
	\end{enumerate}
\end{thm}
\begin{proof}
If the damping parameters in the equation \eqref{se1.4} depends on the time, then rewriting \eqref{se1.4} as
\begin{equation}\label{se5.68}
	(u'',v) + \beta(t)a (u',v) + \alpha(t)(u',v) + a(u,v) = 0, \; \forall \; v \in D(A^{1/2}).
\end{equation}
Differentiate \eqref{se5.68} with respect to time and choose $u' = w$ and $w_\delta = e^{\delta t}w$ to arrive at 
\begin{equation*} \label{se5.59}
	\begin{aligned}
		\left(w_\delta'', v\right) + \beta(t)a(w_\delta', v) &+ \left(1-\beta(t)\delta\right)a(w_\delta, v) + \left(\alpha(t) - 2\delta\right)\left(w_\delta', v\right)
		\\
		 &- \delta\left(\alpha(t)-\delta\right)\left(w_\delta, v\right) + \beta'(t)a(w_\delta, v) + \alpha'(t)\left(w_\delta, v\right) = 0.
	\end{aligned}
\end{equation*}
A choice of $v = w_\delta'$ in \eqref{se5.68} yields
\begin{equation}\label{se5.70}
	\begin{aligned}
		\frac{{\rm d}}{{\rm d}t}&\left(\mathcal{\tilde{E}}(w_\delta)(t) + \frac{\beta'(t)}{2}|w_\delta|^2_1 + \frac{\alpha'(t)}{2}\|w_\delta\|^2\right) + \left(\alpha(t) + \beta(t)\lambda_1 - 2\delta\right)\|w_\delta'\|^2 
		\\
		&\hspace*{2cm} + \frac{\beta'(t)\delta}{2}|w_\delta|^2_1 + \frac{\alpha'(t)\delta}{2}\|w_\delta\|^2 = \frac{\beta''(t)}{2}|w_\delta|^2_1 + \frac{\alpha''(t)}{2}\|w_\delta\|^2 \\
        &\hspace*{2cm} \leq M (1+\lambda_1^{-1}) e^{2\delta t}\;E_{A^{1/2}}(u)(t) \\
        & \hspace*{2cm} \leq C\; E_{A^{1/2}}(u)(0),
	\end{aligned}
\end{equation}
where
\begin{equation*}
	\mathcal{\tilde{E}}(w_\delta)(t) = \half\left(\|w_\delta'\|^2 + \left(1 - \beta(t)\delta\right)|w_\delta|^2_1 - \delta\left(\alpha(t) - \delta\right)\|w_\delta\|^2\right).
\end{equation*}
Due to the assumptions in $\alpha(t), \beta(t)$ and the restriction in $\delta,$ the second, third, and fourth terms on the left-hand side of \eqref{se5.70} are nonnegative. An integration in time with boundedness of $\beta''$ and $\alpha''$ together with the estimate of Theorem ~\ref{thm:a(t):eng} for $j=1/2$ shows
\begin{equation*}
	\begin{aligned}
		\mathcal{\tilde{E}}(w_\delta)(t) &\leq \mathcal{\tilde{E}}(w_\delta)(0) + \frac{\beta'(0)}{2}|w_\delta(0)|^2_1 + \frac{\alpha'(0)}{2}\|w_\delta(0)\|^2+ C \;\mathcal{{E}}(w_\delta)(t)
		\\
		&\leq \mathcal{\tilde{E}}(w_\delta)(0) + \left(\frac{\beta'(0)}{2} + \frac{\alpha'(0)}{2\lambda_1}\right)|w_\delta(0)|^2_1 
        \\
        &\leq C \;\big(\mathcal{\tilde{E}}(w_\delta)(0) + t E_{A^{1/2}}(u)(0).
	\end{aligned}
\end{equation*}
A use of \eqref{se5.63} by replacing $u_{\delta}$ by $w_{\delta}$ completes the rest of the proof.
\end{proof}

Similarly, the decay property of the higher-order derivative with respect to time of the energy function can be derived, but we refrain from giving a proof.

%%%%%%%%%%%%%%%%%%%%%%%%%%%%%%%%%%%%%%%%%%%%%%%%%%%%%%%%%%%%%%%%%%%%%%%%%%%%%%%%%
\subsection{Space varying damping parameters}\label{subse-space-varying}
This subsection deals with equation \eqref{se1.1}, when $\alpha = \alpha(x)$ and $\beta = \beta(x)$  and with the related decay property of the solution.
 \begin{thm}\label{thm:spacdep}
	Let $\alpha = \alpha(x)$ and $\beta = \beta(x)$ in the equation \eqref{se1.1}, i.e.,
		\begin{equation}\label{eq:spadepeq}
			u'' + A^{1/2} (\beta(x) A^{1/2} u) + \alpha(x)u' + Au = 0 \; \mbox{in} \; \Omega\times(0,\infty).
		\end{equation} 
		Assume that
		\[
		0 < \tilde{\alpha}_{\min} \leq \alpha(x) \leq \tilde{\alpha}_{\max}, ~ \mbox{and}~ 0 < \tilde{\beta}_{\min} \leq \beta(x) \leq \tilde{\beta}_{\max} ~ %\mbox{and}~ {\color{red} \nabla_x\cdot\beta(x) \leq 0, %~  \Delta_x\beta(x) \leq 0, ~ \forall~ x \in \Omega.}
		\]
		Then, the following holds true
		\begin{enumerate}[label=(\roman*)]
			\item if $\delta < \min\left(\frac{\tilde{\alpha}_{\min} + \tilde{\beta}_{\min}\lambda_1}{2}, \frac{\lambda_1}{\tilde{\alpha}_{\max} + \tilde{\beta}_{\max}\lambda_1}\right)$ or  ~ $\delta = \frac{\tilde{\alpha}_{\min} + \tilde{\beta}_{\min}\lambda_1}{2}$, then
			\begin{equation*}
				E(u)(t) \leq Ce^{-2\delta t}E(u)(0),
			\end{equation*}
			\item if $\delta = \frac{\lambda_1}{\tilde{\alpha}_{\max} + \tilde{\beta}_{\max}\lambda_1}$, then
			\begin{equation*}
				\|u'\|^2 \leq C E(u)(0).
			\end{equation*}
		\end{enumerate}
	\end{thm}
	\begin{proof}
		If $\alpha$ and $\beta$ depends on the space variable $x$, then the variational formulation of equation \eqref{se1.1} is to seek $u : [0,\infty) \to H^1_0(\Omega)$ such that for all $v \in H^1_0(\Omega)$ and for any $t> 0$
		\begin{equation}\label{se5.76}
			\begin{aligned}
				\left(u'', v\right) + \left(\beta A^\half u', A^\half v\right) 
                %+ \left(\nabla_x\cdot \beta A^\half u', v\right)
                + \left(\alpha u', v\right) + a(u, v) = 0.
			\end{aligned}
		\end{equation}
		Choose $u = e^{-\delta t}u_\delta$ in \eqref{se5.76}, and rewrite it as
		\begin{equation*}\label{eq:5.38}
			\begin{aligned}
				\left(u''_\delta, v\right) &+ \left(\beta A^\half u'_\delta, A^\half v\right) + \left(\left(\alpha - 2\delta\right)u'_\delta, v\right) + \left(\left(1 - \beta\delta\right)A^\half u_\delta, A^\half v\right) 
				\\
				&\qquad\qquad\qquad\qquad
                %+ \left(\nabla_x\cdot\beta A^\half u'_\delta, v\right) %- \delta\left(\nabla_x\cdot\beta A^\half u_\delta, %v\right) 
                - \delta\left(\left(\alpha - \delta\right)u_\delta, v\right) = 0,
			\end{aligned}
\end{equation*}
with $ u_\delta(x,0) = u_0 \; \mbox{and} \; (u_\delta)'(x,0) = \delta u_0 + u_1.$ A choice of $v = u'_\delta$ in the above equation, yields
\begin{equation}\label{eqn:enginx}
			\begin{aligned}
			\frac{{\rm d}}{{\rm d}t}\mathcal{\tilde{E}}_x(u_\delta)(t) + \left(\tilde{\alpha}_{\min} + \tilde{\beta}_{\min}\lambda_1 - 2\delta\right)\|u'_\delta\|^2 \leq 0, 
			\end{aligned}
		\end{equation}
		where 
		\begin{equation*}
			\begin{aligned}
				\mathcal{\tilde{E}}_x(u_\delta)(t) = \half\left(\|u'_\delta\|^2 + \int_{\Omega}\left(1 - \beta\delta\right)\|A^\half u_\delta\|^2\,{\rm d}x - \delta\int_{\Omega} \left(\alpha - \delta\right)\|u_\delta\|^2\,{\rm d}x\right).
			\end{aligned}
		\end{equation*}
		Taking integration of \eqref{eqn:enginx} in time, we obtain
		\begin{equation*}
			\mathcal{\tilde{E}}_x(u_\delta)(t) \leq \mathcal{\tilde{E}}_x(u_\delta)(0)
		\end{equation*}
	provided $\delta \leq \frac{\tilde{\alpha}_{\min} + \tilde{\beta}_{\min}\lambda_1}{2}.$ 
    A use of \eqref{poincare} yields
    
	\begin{equation*}\label{eq:matex}
		\begin{aligned}
			\mathcal{\tilde{E}}_x(u_\delta)(t) \geq \half\left(\|u'_\delta\|^2 + \left(1 - \frac{\tilde{\alpha}_{\max} + \tilde{\beta}_{\max}\lambda_1}{\lambda_1}\delta\right)|u_\delta|^2_1 + \delta^2\|u_\delta\|^2\right)	
			> \gamma_0 E(u_\delta)(t),
		\end{aligned}
	\end{equation*}  
	provided $\delta < \gamma_0$, where $\gamma_0 = \frac{\lambda_1}{\tilde{\alpha}_{\max} + \tilde{\beta}_{\max}\lambda_1}$. This completes the proof.	
\end{proof}

\begin{thm}\label{thm:spahigeng}
	Under the assumption of Theorem ~\ref{thm:spacdep} and for $k = 2,3,\dots$. The following holds
	\begin{enumerate}[label=(\roman*)]
		\item if $\delta < \min\left(\frac{\tilde{\alpha}_{\min} + \tilde{\beta}_{\min}\lambda_1}{2}, \frac{\lambda_1}{\tilde{\alpha}_{\max} + \tilde{\beta}_{\max}\lambda_1}\right)$ or ~ $\delta = \frac{\tilde{\alpha}_{\min} + \tilde{\beta}_{\min}\lambda_1}{2}$, then 
		\begin{equation*}
			E(u^{(k)})(t) \leq Ce^{-2\delta t}E(u^{(k)})(0),
		\end{equation*}
		\item if $\delta = \frac{\lambda_1}{\tilde{\alpha}_{\max} + \tilde{\beta}_{\max}\lambda_1}$, then
		\begin{equation*}
			\|u^{(k)}\|^2 \leq C E(u^{(k)})(0).
		\end{equation*}
	\end{enumerate}
\end{thm}
	
\begin{proof}
	On differentiating the equation \eqref{eq:spadepeq} $(k-1)$ times in time, we arrive at 
	\begin{equation*}
		u^{(k+1)} + A^{1/2}(\beta A^{1/2}u^{(k)}) + \alpha u^{(k)} + Au^{(k)} = 0.
	\end{equation*}
	Choose $u^{(k-1)} = w$ and proceeding in the same manner as in the proof of Theorem~\ref{thm:spacdep}, we obtain the desired results. This completes the proof.
\end{proof}	
For the two Subsections \ref{subse-time-varying}--\ref{subse-space-varying}, we refrain from the error analysis of finite element approximation as the proof goes in parallel to the proofs in Sections~\ref{sec:semd}--\ref{sec:fulld}.
\begin{Rem}
Our analysis will include the case when 
 $$A v=-\sum_{l,q=1}^{d} \frac{\partial}{ \partial x_l}( a_{lq}(x)\frac{\partial v}{\partial x_q})+a(x)v,   
$$ 
where the $d\times d$ matrix  $[a_{lq}(x)]$ is symmetric, uniformly  positive definite and $a\geq 0.$
\end{Rem}
\begin{Rem}
The present analysis can be easily extended to the problem \eqref{se1.1} when $A=\Delta^2$ with clamped, hinged, or free homogeneous boundary conditions.
\end{Rem}

\subsection{Abstract discrete problem}
Given two finite-dimensional spaces $V_h$ and $H_h $ with inner products and norms defined, respectively, as $(\cdot,\cdot)_{V_h}, (\cdot,\cdot)_{H_h}$ and $ \|\cdot\|_{V_h},\;\|\cdot\|_{H_h},$ such that 
$V_h$ is continuously imbedded in $H_h$ and the discrete symmetric bounded, positive definite linear operator $A_h:V_h\mapsto V_h$ and related discrete bilinear form $a_h(\cdot,\cdot): V_h\times V_h\mapsto R$ defined by $a_h(v_h,w_h)= (A_h v_h, w_h)_{H_h}= (A^{1/2}_h v_h,A^{1/2}_h w_h)_{H_h},$ the semidiscrete damped wave equation is to find $u_h(t) \in V_h$ for $t>0$ such that 
\begin{eqnarray}\label{eq:abstract}
    u_h'' + \beta A_h u_h' + \alpha u_h' + A_h u_h=0, \;\; t>0, 
\end{eqnarray}
with $u_h(0)= u_{0,h}\in V_h$ and $u_h'(0)= u_{1,h}\in H_h,$ where $u_h(0)$ and $u_h'(0)$ known functions. As long as the following discrete version of Poincare inequality 
\begin{equation*}\label{discrete:Poinacare}
    \|A^{1/2}_h v_h\|^2_{H_h} \geq \lambda \|v_h\|^2_{H_h}
\end{equation*}
is satisfied with $\lambda>0$ independent of the discretization parameter $h,$ it is possible to prove the decay estimates for the semidiscrete solution $u_h,$ where the decay estimates are uniform with respect to $h.$

\subsubsection{Finite Difference Schemes}
We assume that the domain $\Omega$ is the unit square in $\mathbb R^2$ and $A v=-\Delta v$ with homogeneous Dirichlet boundary condition. For the numerical solution of \eqref{se1.1} with $\Omega= (0,1)\times (0,1)$, we select a mesh of width $h=\frac{1}{M}$, where $M$ is a positive integer, and cover $\bar\Omega=~\Omega\cup\partial\Omega$ with a square grid of mesh points $x_{ij}=(ih,jh)$ for $i,~j = 0, 1, \ldots, M$. Let $\Omega_h=\{x_{ij}:~x_{ij} \in \partial \Omega\}$. We can cover the whole $\mathbb R^2$ with such a square grid, and will denote it by $\mathbb R_h^2.$

For a function $w$ defined on $\mathbb R_h^2$, we adopt the following notation: for $x~\in~\mathbb R_h^2$ and $l~=~1, 2$, we have
\begin{eqnarray}
\begin{aligned}
 w^{\pm l} = w(x \pm h {\bf e}_l),\; w^{+ l,-q}  =  w(x &+ h {\bf e}_l-h {\bf e}_q), \; 
 \nabla_l w(x) =  \frac{w(x+h{\bf e}_l)-w(x)}{h} \; \mbox{and} \nonumber\\
 \bar\nabla_l w(x)  &=  \frac{w(x)-w(x-h{\bf e}_l)}{h},\nonumber
 \end{aligned}
 \end{eqnarray}
 where ${\bf e}_l$ is the $l$th unit vector in $\mathbb R^2$.

  Let $\mathcal D_h$ denote the mesh functions defined on $\mathbb R_h^2$ that vanish outside of $\Omega_h$. For $u,~v\in~ \mathcal D$, we now introduce the discrete $L^2$-space, denoted by $L_h^2$ with inner product and norm given, respectively, by
  \begin{equation*}
   <w,v>  =  h^2 \sum_{x\in{\mathbb R}_h^2}w(x)v(x)\;\; \mbox{and}\;\;
   \|w\|_{0,h} = <w,w>^{\frac{1}{2}} .
  \end{equation*}
  Moreover, for $w~\in~ {\mathcal D}_h$, let $H_h^1$ denote the distance analogue of the $H_0^1$-Sobolev space with norm $\|w\|_{1,h}^2=\sum_{l=1}^2 \|\nabla_l w\|_{0,h}^2$. We also introduce the discrete $H^2\cap H_0^1$-Sobolev space with norm
\begin{equation}
 \|w\|_{2,h}^2=\|w\|_{1,h}^2+\sum_{l=1,q}^2 \|\nabla_l\bar\nabla_q \nonumber w\|_{0,h}^2,\,\,\,w\in{\mathcal D_h}
\end{equation}
and denote it by $H_h^2$. Whenever there is no confusion, we write $\|w\|$ and $\|w\|_j,~j~=~1,2$, instead of $\|w\|_{0,h}$ and $\|w\|_{j,h}.$ 

For functions $v$ and $w$ in ${\mathcal D}_h$, the following identities are easy consequences of summation by parts 
\begin{equation}\label{eq:discrete-summation}  
 <\nabla _l v,w>= -<v,\bar\nabla_l w>, \,\,\,l=1,2.
\end{equation}

Let $A_h$ denote the finite difference operator approximating $A$, which is defined for $v_h\in H^2_h$ by
\begin{equation*}
 A_h v_h=- \Delta_h v_h= -\sum_{l,q=1}^{d} \bar\nabla_l(\nabla_q v_h),
 %+\bar\nabla_l(a_{lq}(t)\bigtriangledown_q V)]+S(a(t))V,
\end{equation*}
and with $u_h(t)\in H_h^1,$ the semidiscrete finte difference formulation is given by \eqref{eq:abstract}, where $V_h=H^1_h$ and $H_h=L^2_h.$
Setting for $v_h, w_h \in H_h^1$, $a(v_h,w_h) = \sum_{l=1}^d <\nabla_l v_h, \nabla_l w_h>= <A_h v_h,w_h >$ by \eqref{eq:discrete-summation},
the following properties are valid.
\begin{itemize}
 \item[(i)]  $a(v_h,v_h)= <A_h v_h,v_h > =\|v_h\|^2_{H^1_h}.$
 \item[(ii)] 
 $|a(v_h,w_h)|= |<A_h(t)V,W>| \leq C \|v_h\|_{1,h}\|w_h\|_{1,h}.$
\end{itemize}
 Since $A_h$ is symmetric and positive definite, $A^{1/2}$ is well defined and 
 $$a(v_h,w_h) = <A_h^{1/2} v_h, A^{1/2} w_h> = <A_h v_h,w_h >$$ and $$a(v_h,v_h) = \|A^{1/2} v_h\|^2_{L^2_h}.$$
From the discrete version of Poinc\'are  inequality it follows that there exists a positive constant $\lambda>0$ independent of $h$ such that for $v\in H^1_h$, we have
$$ \lambda \|v_h\|^2~\leq \sum_{l=1}^{d}\|\nabla_l V\|^2 =\|A^{1/2} v_h\|_{0,h}^2.$$

\begin{Rem}
 This analysis in abstract discrete problem can be easily extended to spectral method applied to  the problem \eqref{se1.1}--\eqref{se1.3}.

 %Moreover, for symmetric interior penalty discontinuous Galerkin 
 %method, the present analysis can be extended
\end{Rem}

%%%%%%%%%%%%%%%%%%%%%%%%%%%%%%%%%%%
\section{Numerical Experiments}
\setcounter{equation}{0}
%%%%%%%%%%%%%%%%%%%%%%%%%%%%%%%%
This section focuses on some numerical experiments, whose results confirm our theoretical findings. We consider the following strongly damped wave equation:
\begin{equation*}
    u'' - \beta \, \Delta u' + \alpha \, u' - \Delta u = 0, \; (x, y) \in \Omega, \; t > 0.
\end{equation*}

In all Examples \ref{Exm1}-\ref{Exm3}, the equations are solved using our completely discrete scheme up to the final time $T = 1.0$ with time step $k = 2 h^2$ to minimize the effect of time discretization error. The numerical experiments are performed using FreeFem++ software with piecewise linear elements.

In each case, the experimental convergence rate of the error is computed using
\[
\mbox{Rate} = \frac{\log(E_{h_{i}})-\log (E_{h_{i+1}})}{\log (\frac{h_i}{h_{i+1}})},
\]
where $E_{h_{i}}$ denotes the error norm using $h_i$ as the spatial discretization parameter at the $i$th stage. 
\begin{Exm}\label{Exm1}
    For the first example, we consider $\alpha = \pi, \beta =  \displaystyle{\frac{1}{\pi}}$ and $\Omega = (0, 1) \times (0, 1)$ such that the exact solution is given by 
    $$
    u(x, y, t) = e^{- \pi t} \sin (\pi x) \sin(\pi y).
    $$
    The corresponding initial conditions are
    \beas 
    u(x, y, 0) = \sin (\pi x) \sin(\pi y), \; u'(x, y, 0) = - \pi \sin (\pi x) \sin(\pi y), \; (x, y) \in \Omega,
    \eeas
    and the boundary condition is
    \beas
    u = 0, \; (x, y) \in \partial \Omega, \; t > 0.
    \eeas
%Choose $\alpha$ and $\beta$ such that the exact solution is given by
\end{Exm}
%{\bf Example 1.} 
%In the Table 1, we discuss the errors and rate of convergence in $L^2, \; L^{\infty}$ and $H^1$- norms and it is observed that convergence rates confirm the theoretical convergence rate. \\
Table \ref{tab:ex1} shows the error and rate of convergence in $L^2, \; L^{\infty}$ and $H^1$-norms, confirming our theoretical findings. 

\begin{table}[h!]
    \centering
    \begin{tabular}{|c|cccccc|}
    \hline
         & & $\alpha = \pi$  && $\beta =  \displaystyle{\frac{1}{\pi}}$ && \\ \cline{2-7}
		N & $\|u-U\|_{L^{2}}$ &Rate & $\|u-U\|_{L^{\infty}}$  &  Rate&$\|u-U\|_{H^{1}}$&Rate
		\\  \hline
		5 & 8.8349(-3)& &          7.1418(-3)  &       & 6.4637(-2)&  \\
		10 & 2.1124(-3)& 2.0643 &  1.7120(-3) & 2.0606 & 2.5999(-2)& 1.3139 \\
		15 & 9.5284(-4)& 1.9635 & 7.7232(-4) & 1.9632 & 1.7282(-2)& 1.0072  \\
		20 & 5.3192(-4)& 2.0264 & 4.3115(-4) & 2.0263 & 1.2763(-2)& 1.0538 \\ 
		25 & 3.4187(-4)& 1.9810 & 2.7711(-4) & 1.9810 & 1.0273(-2)& 0.9727  \\
		30 & 2.3671(-4)& 2.0163 & 1.9187(-4) & 2.0163 & 8.5325(-3)& 1.0180  \\
		\hline
    \end{tabular}
    \caption{Example \ref{Exm1}:  $L^2$, $L^{\infty}$, $H^1$ errors and convergence rates for $U$.}
    \label{tab:ex1}
\end{table}
In Figure \ref{ex1-f1}, we show the plots of discrete energy $E_h(u_h)$ and continuous energy $E(u)$. In Figure \ref{ex1-f2}(a), we observe the decay behaviour of discrete energies, and Figure \ref{ex1-f2}(b) shows the numerically computed decay rates.  

\begin{figure}[h!]
\centering
  \includegraphics[width=16.2cm,height=5.8cm]{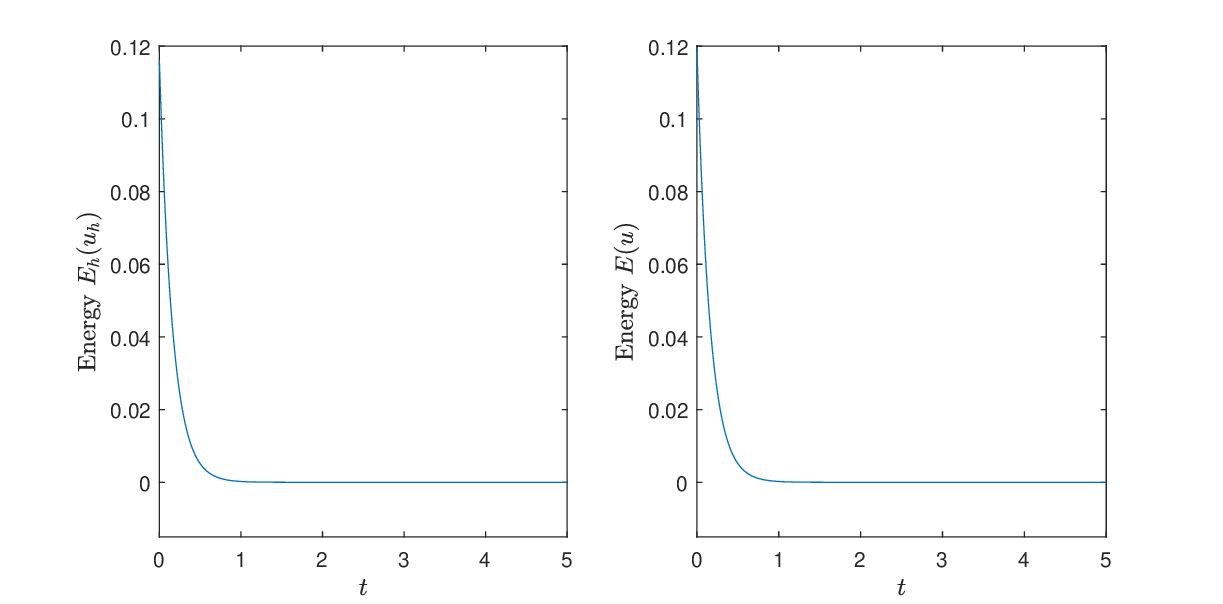}
  \caption{Example \ref{Exm1}: The plots of discrete energy $E_h(u_h)$ and continuous energy $E(u)$.}
  \label{ex1-f1}
\end{figure}

\begin{figure}[h!]
\centering
\begin{subfigure}{.45\textwidth}
  \centering
  \includegraphics[width=7.3cm, height=5.8cm]{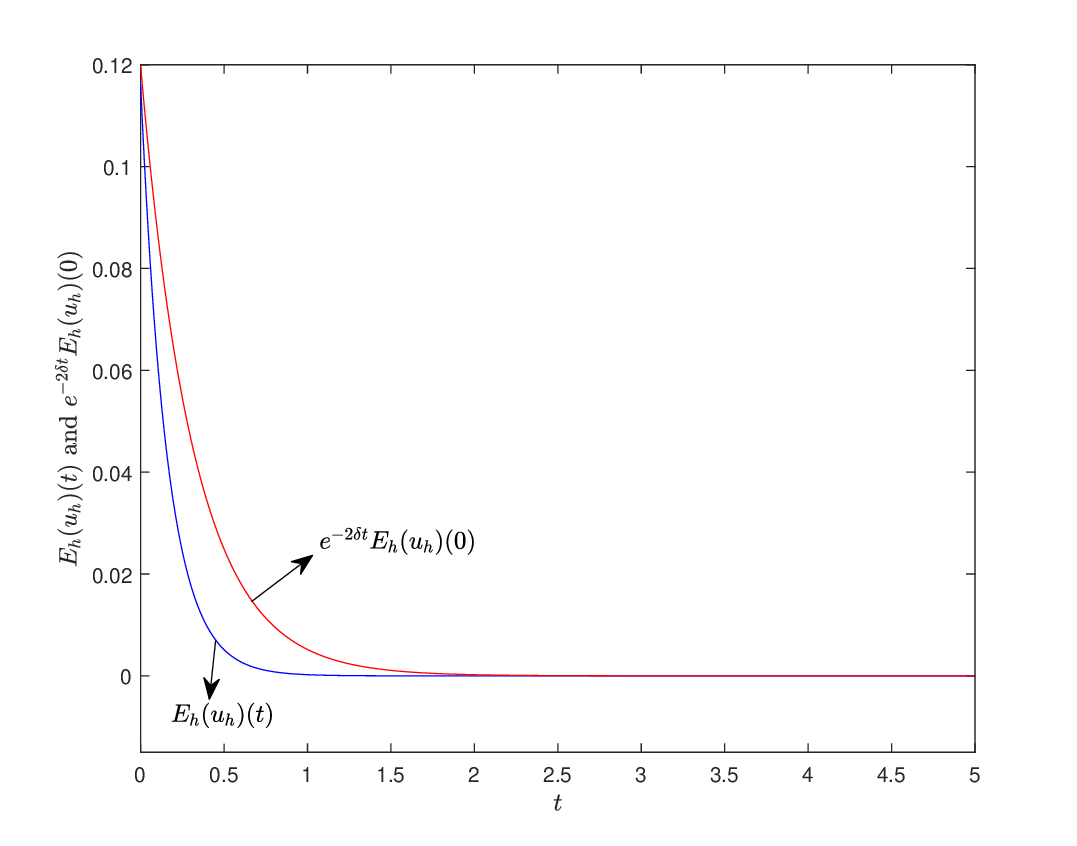}
  \caption{The decay behaviour of discrete Energies}
  \label{fig:Ex1.1}
\end{subfigure}
\begin{subfigure}{.45\textwidth}
  \centering
  \includegraphics[width=7.3cm, height=5.8cm]{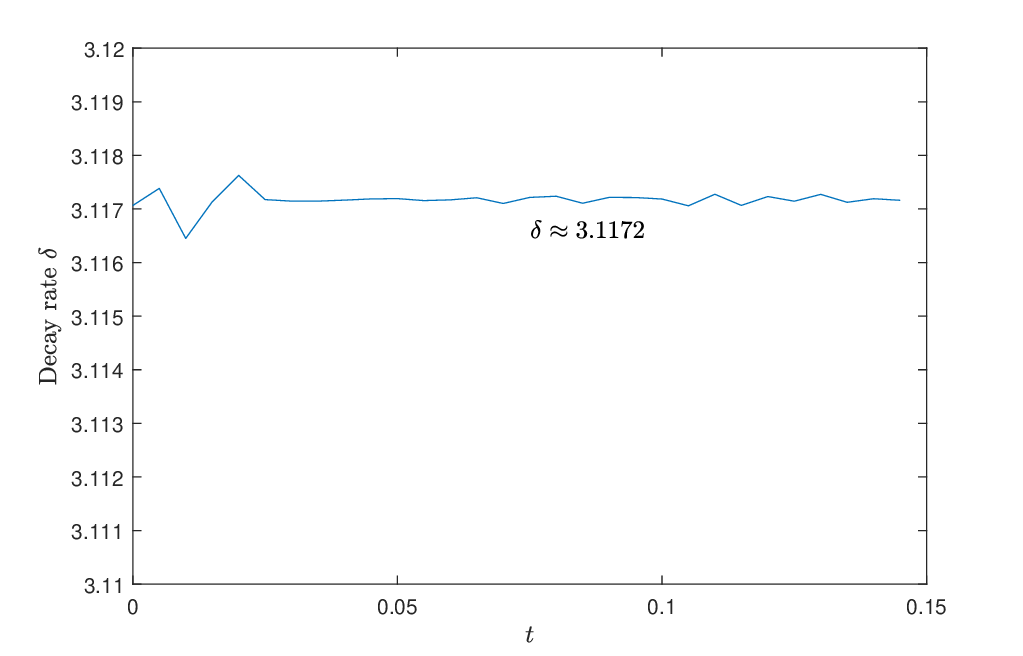}
  \caption{Numerically computed decay rate $\delta$ }
  \label{fig:Ex1.2}
\end{subfigure}
\caption{Example \ref{Exm1}: Discrete energy decay profiles and estimate of decay rate.  }
\label{ex1-f2}
\end{figure}

%%%%%%%%%%%%%%%%%%%%%%%%%%%%%%%%%%%%%%%%%%%%%%%%%%%%%%%%%%%%%%%%%%%%%
%%%%%%%%%%%%%%%%%%%%%%%%%%%%%%%%%%%%%%%%%%%%%%%%%%%%%%%%%%%%%%%%%%%%%

\begin{Exm}\label{Exm2}
    For the second example, take $\alpha = \displaystyle{\frac{\left(\pi^2 + 4 \right)}{2\pi}}, \; \beta =  \displaystyle{\frac{\pi}{4}}$ and $\Omega = (0, \pi) \times (0, \pi)$. The corresponding exact solution is 
    $$
    u(x, y, t) = e^{-\pi t} \sin (x) \sin(y).
    $$
    Here, the initial conditions are
    \beas
    u(x, y, 0) = \sin (x) \sin(y), \; u'(x, y, 0) = - \pi \sin (x) \sin(y), \; (x, y) \in \Omega 
    \eeas
    and the boundary condition is
    \beas
    u = 0, \; (x, y) \in \partial \Omega, \; t > 0.
    \eeas
\end{Exm}

In Table \ref{tab:ex2}, we show the errors and rate of convergences in $L^2, \; L^{\infty}$ and $H^1$-norms that validate our theoretical results established in the previous sections. \\
%\vspace{2cm}
\begin{table}[h!]
    \centering
    \begin{tabular}{|c|cccccc|}
    \hline
    & & $\alpha = \displaystyle{\frac{\left(\pi^2 + 4 \right)}{2\pi}}$  && $\beta =  \displaystyle{\frac{\pi}{4}}$ && \\ \cline{2-7}
		$N$ & $\|u-U\|_{L^{2}}$ & Rate & $\|u-U\|_{L^{\infty}}$  &  Rate & $\|u-U\|_{H^{1}}$ & Rate
		\\ \hline
		5 & 5.1151(-2)& &          1.3003(-1)  &       & 1.0401(-1)&  \\
		10 & 1.3260(-2)& 1.9477 &  3.3688(-2) & 1.9485  & 3.4207(-2)& 1.6044 \\
		15 & 6.0028(-3)& 1.9546 & 1.5239(-2) & 1.9564 & 2.0064(-2)& 1.3158  \\
		20 & 3.3779(-3)& 1.9986 & 8.5727(-3) & 1.9998 & 1.4008(-2)& 1.2490 \\ 
		25 & 2.1703(-3)& 1.9826 & 5.5071(-3) & 1.9833 & 1.0926(-2)& 1.1136  \\
		30 & 1.5065(-3)& 2.0022 & 3.8225(-3) & 2.0026 & 8.9166(-3)& 1.1145  \\
		\hline
    \end{tabular}
    \caption{Example \ref{Exm2}:  $L^2$, $L^{\infty}$, $H^1$ errors and convergence rates for $U$.}
    \label{tab:ex2}
\end{table}
In Figure \ref{ex2-f1}, we show the plots of discrete energy $E_h(u_h)$ and continuous energy $E(u)$. From Figure \ref{ex2-f2}(a), we observe the decay behaviour of discrete energies and the numerically computed decay rate in Figure \ref{ex2-f2}(b). 

\begin{figure}[h!]
\centering
  \includegraphics[width=16.2cm,height=5.8cm]{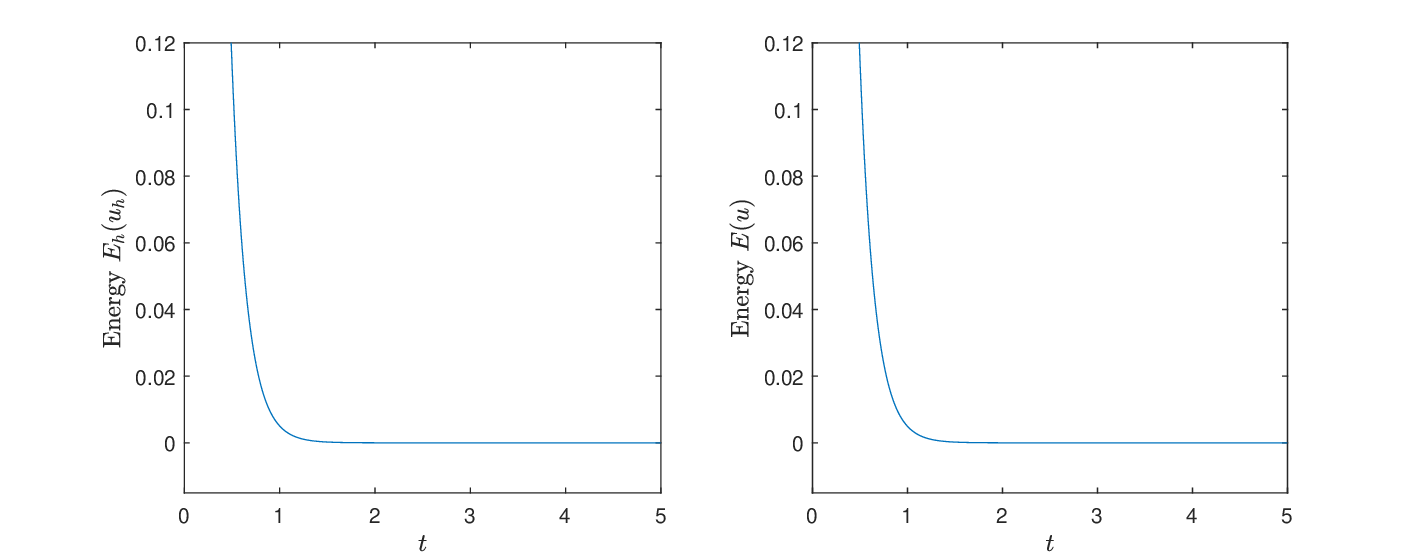}
  \caption{Example \ref{Exm2}: The plots of discrete energy $E_h(u_h)$ and continuous energy $E(u)$.}
  \label{ex2-f1}
\end{figure}

\begin{figure}[h!]
\centering
\begin{subfigure}{.45\textwidth}
  \centering
  \includegraphics[width=7.3cm, height=5.8cm]{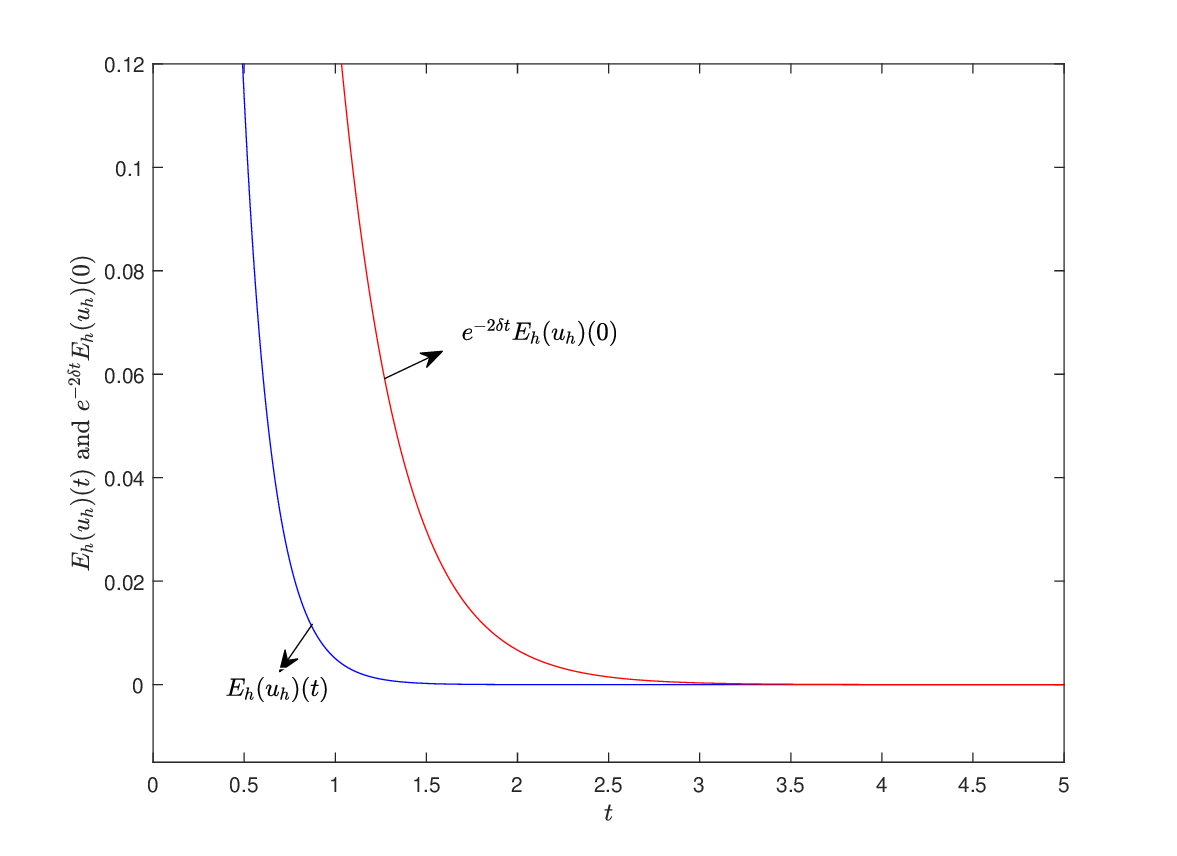}
  \caption{The decay behaviour of discrete Energies}
  \label{fig:Ex2.1}
\end{subfigure}
\begin{subfigure}{.45\textwidth}
  \centering
  \includegraphics[width=7.3cm, height=5.8cm]{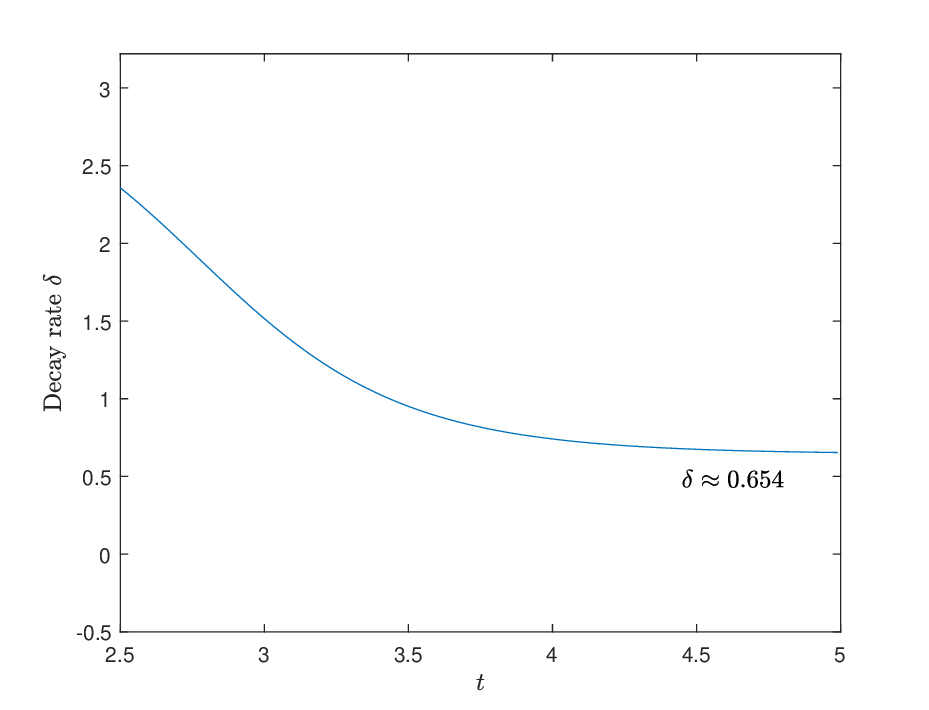}
  \caption{Numerically computed decay rate $\delta$ }
  \label{fig:Ex2.2}
\end{subfigure}
\caption{Example \ref{Exm2}: Discrete energy decay profiles and estimate of decay rate.  }
\label{ex2-f2}
\end{figure}

%%%%%%%%%%%%%%%%%%%%%%%%%%%%%%%%%%%%%%%%%%%%%%%%%%%%%%%%%%%%%%%%%%%%%
%%%%%%%%%%%%%%%%%%%%%%%%%%%%%%%%%%%%%%%%%%%%%%%%%%%%%%%%%%%%%%%%%%%%%
\noindent
In the next example, we discuss the errors and convergence rates for the two cases: 
\begin{enumerate}
	\item[(i)] when $\alpha \neq 0$ and $\beta = 0$, 
	\item[(ii)] when $\alpha = 0$ and $\beta \neq 0$.
\end{enumerate} 
\begin{Exm}\label{Exm3}
    The third example has $\Omega = (0, \pi) \times (0, \pi)$ and for the $\alpha,\beta$ the following two cases: 
    \begin{enumerate}[label=(\roman*)]
	\item $\alpha = \displaystyle{\frac{\left(\pi^2 + 2 \right)}        {\pi}}$ and $\beta = 0$. 
	\item $\alpha = 0$ and $\beta = \displaystyle{\frac{\left(\pi^2     + 2 \right)}{2\pi}}$.
    \end{enumerate} 
    For both cases, the exact solution is
    $$
    u(x, y, t) = e^{-\pi t} \sin ( x) \sin( y).
    $$
    The boundary condition is
    \beas
    u = 0, \; (x, y) \in \partial \Omega, \; t > 0,
    \eeas
    and the initial conditions are
    \beas
    u(x, y, 0) = \sin ( x) \sin( y), \quad u'(x, y, 0) = -  \pi \sin ( x) \sin( y), \; (x, y) \in \Omega. 
    \eeas
\end{Exm}

In Table \ref{tab:exm31}(i), we show the errors and convergence rates in the $L^2, \; L^{\infty}$ and $H^1$ norms, which validates the theoretical findings established in the previous sections. \\

\begin{table}[h!]
    \centering
    \begin{tabular}{|c|cccccc|}
         \hline
         & & $\alpha = \displaystyle{\frac{\left(\pi^2 + 2 \right)}{\pi}}$  && $\beta = 0$ && \\ \cline{2-7}
		$N$ & $\|u-U\|_{L^{2}}$ & Rate & $\|u-U\|_{L^{\infty}}$  &  Rate & $\|u-U\|_{H^{1}}$ & Rate
		\\ \hline
		5 & 1.1860(-2)& &          2.8532(-2)  &       & 5.1691(-2)&  \\
		10 & 2.3853(-3)& 2.3139 &  5.7294(-3) & 2.3162  & 2.3811(-2)& 1.1183 \\
		15 & 1.0850(-3)& 1.9429 & 2.6289(-3) & 1.9214 & 1.6598(-2)& 0.8901  \\
		20 & 5.9501(-4)& 2.0882 & 1.4474(-3) & 2.0745 & 1.2468(-2)& 0.9945 \\ 
		25 & 3.8420(-4)& 1.9603 & 9.3636(-4) & 1.9518 & 1.0120(-2)& 0.9349  \\
		30 & 2.6438(-4)& 2.0499 & 6.4515(-4) & 2.0432 & 8.4439(-3)& 0.9933  \\
		\hline
    \end{tabular}
    \caption{Example \ref{Exm3}$(i)$:  $L^2$, $L^{\infty}$, $H^1$ errors and convergence rates for $U$.}
    \label{tab:exm31}
\end{table}

\begin{figure}[h!]
\centering
  \includegraphics[width=16.2cm, height=5.8cm]{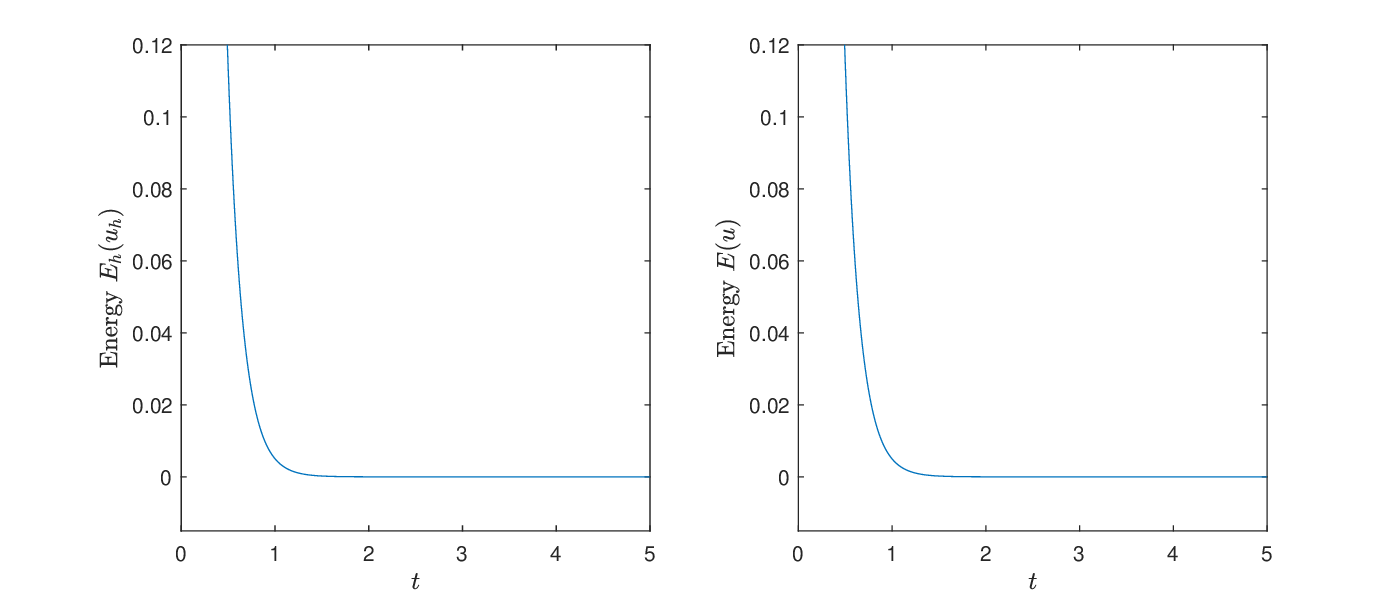}
  \caption{Example \ref{Exm3}$(i)$: The plots of discrete energy $E_h(u_h)$ and continuous energy $E(u)$.}
  \label{ex3-f1}
\end{figure}

\begin{figure}[h!]
\centering
\begin{subfigure}{.45\textwidth}
  \centering
  \includegraphics[width=7.3cm, height=5.8cm]{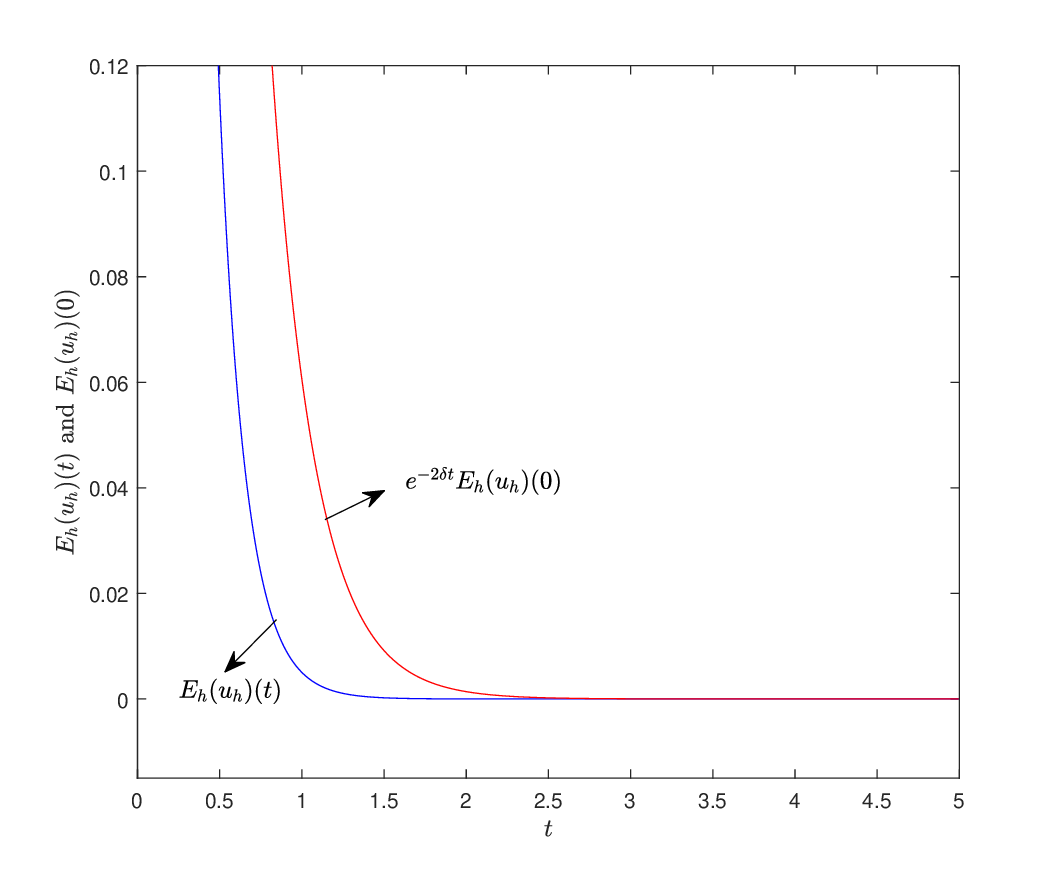}
  \caption{The decay behaviour of discrete Energies}
  \label{fig:Ex3.1}
\end{subfigure}
\begin{subfigure}{.45\textwidth}
  \centering
  \includegraphics[width=7.3cm, height=5.8cm]{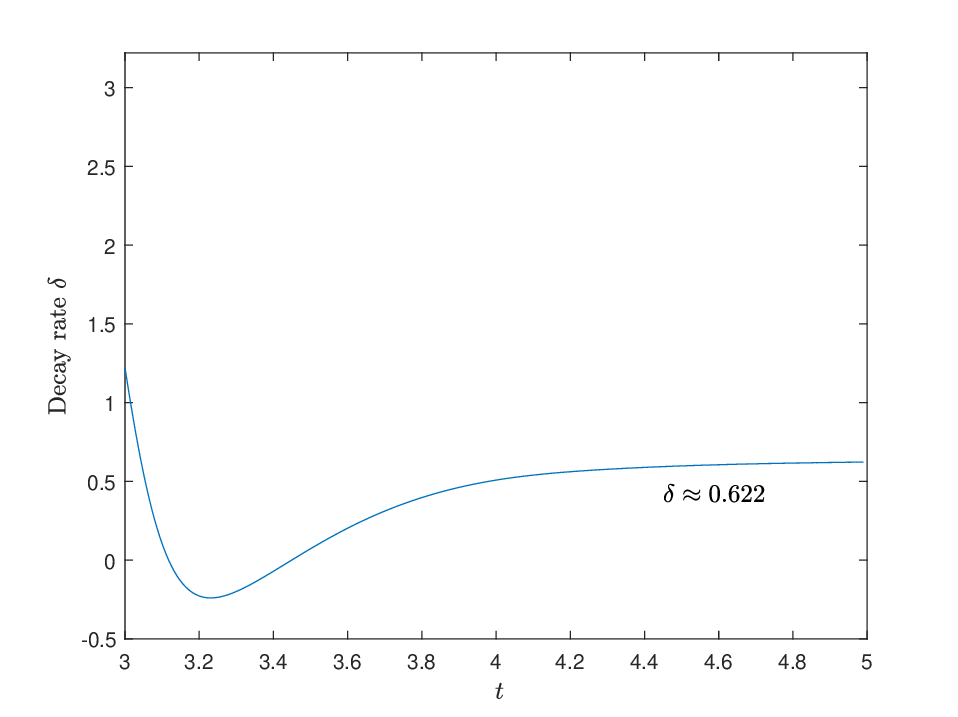}
  \caption{Numerically computed decay rate $\delta$ }
  \label{fig:Ex3.2}
\end{subfigure}
\caption{Example \ref{Exm3}$(i)$: Discrete energy decay profiles and estimate of decay rate for $\alpha \neq 0$ and $\beta = 0$.}
\label{ex3-f2}
\end{figure}

%%%%%%%%%%%%%%%%%%%%%%%%%%%%%%%%%%%%%%%%%%%%%%%%%%%%%%%%%%%%%%%%%%%%
%%%%%%%%%%%%%%%%%%%%%%%%%%%%%%%%%%%%%%%%%%%%%%%%%%%%%%%%%%%%%%%%%%%%

\begin{table}[h!]
    \centering
    \begin{tabular}{|c|cccccc|}
         \hline
	   & & $\alpha = 0$  && $\beta = \displaystyle{\frac{\left(\pi^2    + 2 \right)}{2\pi}}$ && \\ \cline{2-7}
		$N$ & $\|u-U\|_{L^{2}}$ & Rate & $\|u-U\|_{L^{\infty}}$  &    Rate & $\|u-U\|_{H^{1}}$ & Rate
		\\ \hline
		5 & 1.0330(-1)& &          2.6260(-1)  &       &              1.8845(-1)&  \\
		10 & 2.8634(-2)& 1.8511 &  7.2732(-2) & 1.8522  &             5.6827(-2)& 1.7295 \\
		15 & 1.3045(-2)& 1.9390 & 3.3108(-2) & 1.9410 & 2.8944(-2)&   1.6639  \\
		20 & 7.3836(-3)& 1.9784 & 1.8733(-2) & 1.9796 & 1.8353(-2)&   1.5836 \\ 
		25 & 4.7464(-3)& 1.9802 & 1.2040(-2) & 1.9809 & 1.3331(-2)&   1.4328  \\
		30 & 3.3007(-3)& 1.9925 & 8.3719(-3) & 1.9930 & 1.0382(-2)&   1.3710  \\
		\hline
    \end{tabular}
    \caption{Example \ref{Exm3}$(ii)$:  $L^2$, $L^{\infty}$, $H^1$ errors and convergence rates for $U$.}
    \label{tab:exm32}
\end{table}

\begin{figure}[h!]
\centering
  \includegraphics[width=16.2cm, height=5.3cm]{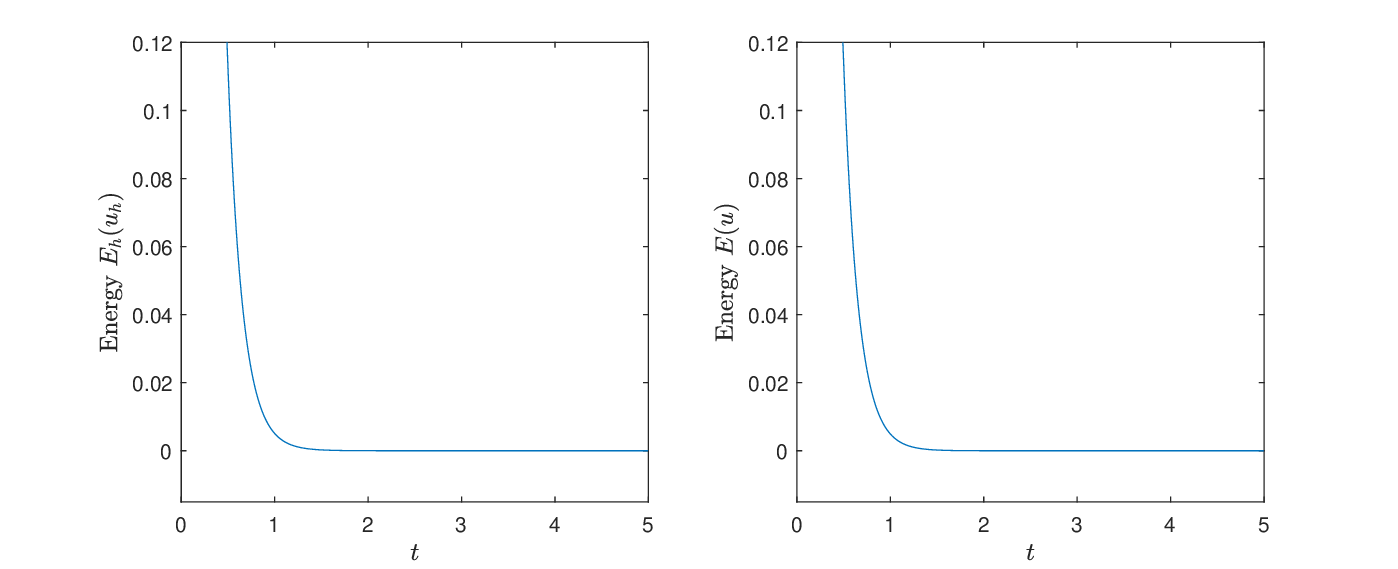}
  \caption{Example \ref{Exm3}$(ii)$: The plots of discrete energy $E_h(u_h)$ and continuous energy $E(u)$.}
  \label{ex4-f1}
\end{figure}

\begin{figure}[h!]
\centering
\begin{subfigure}{.45\textwidth}
  \centering
  \includegraphics[width=6.9cm, height=5.3cm]{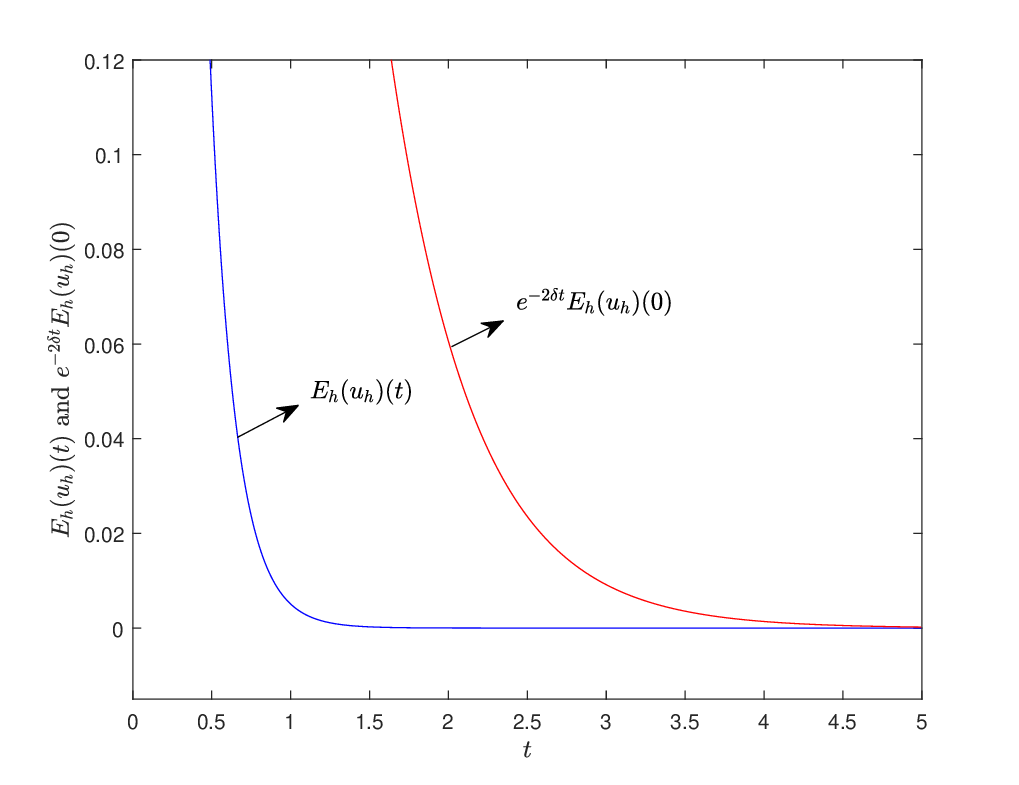}
  \caption{The decay behaviour of discrete Energies}
  \label{fig:Ex4.1}
\end{subfigure}
\begin{subfigure}{.45\textwidth}
  \centering
  \includegraphics[width=6.9cm, height=5.3cm]{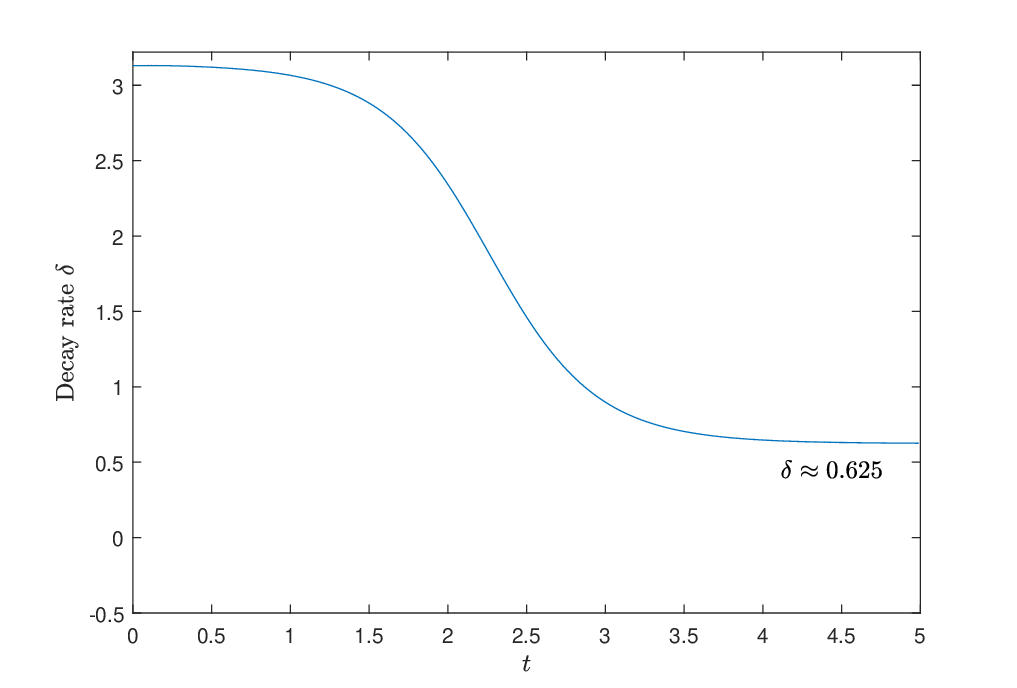}
  \caption{Numerically computed decay rate $\delta$ }
  \label{fig:Ex4.2}
\end{subfigure}
\caption{Example \ref{Exm3}$(ii)$: Discrete energy decay profiles and estimate of decay rate for $\alpha = 0$ and $\beta \neq 0$.  }
\label{ex4-f2}
\end{figure}

\textbf{Observations:} Based on the numerical experiments, we have the following observations.
\begin{enumerate}[label=(\roman*)]
    \item From Tables \ref{tab:ex1}-\ref{tab:exm32} one can observe that by choosing different values of the $\alpha$ and $\beta$, we obtain the order of convergence $2$ for the $L^2$-norm and $1$ for the $H^1$-norm for linear polynomials which confirms  our results in Theorems \ref{thm:H1est}-\ref{thm:L2est}.
    
    \item We present graphs of the energies obtained from the discrete solution, see the Figures \ref{ex1-f1} (Left), \ref{ex2-f1} (Left), \ref{ex3-f1} (Left) and \ref{ex4-f1} (Left) and the exact solution, see the Figures \ref{ex1-f1} (Right), \ref{ex2-f1} (Right), \ref{ex3-f1} (Right) and \ref{ex4-f1} (Right) by taking the different values of $\alpha$ and $\beta$ and degree of polynomial is $1$. It is observed that the continuous energy and discrete energy decay exponentially, and this validates Theorems \ref{thm:eng} and \ref{thm:diseng}.
    
    \item In Figures \ref{ex1-f2}(Left), \ref{ex2-f2}(Left), \ref{ex3-f2}(Left) and \ref{ex4-f2}(Left), we compute the discrete energy for the different time and exponential power minus two delta times into the discrete initial energy for different times, which validates our findings in  Theorem \ref{thm:fulene}.
\end{enumerate}

\section{Conclusion}
This paper deals with a detailed, mathematically rigorous study of semidiscrete and completely discrete approximations for damped wave equations involving weak and strong damping parameters, which preserve uniform exponential decay and optimal error estimates. The analysis involves energy-based techniques to establish consistency between continuous and discrete decay rates and explicit rates that depend on damping parameters and the principal eigenvalue of the associated linear elliptic operator. The rigorous statements and careful estimates significantly improve our results. Including multiple scenarios, such as space, time-dependent damping parameters, and nonhomogeneous source terms, is a major highlight of our approach. The set of numerical experiments, the results of which confirm our theoretical findings, is presented. The generalizations given in the Section $5$ include the problem with nonhomogeneous forcing function, with time and space varying damping coefficients, vibration of beam with damping and an abstract discrete problem with application to both finite difference and spectral approximations showing uniform decay behavior.

\section*{Acknowledgments}
 The authors acknowledge  the financial support by SERB, Govt. India,  via the Project No. CRG/2023/000721.

%%%%%%%%%%%%%%%%%%%%%%%%%
%%%%%%%%%%% BIBLIOGRAPHY %%%%%%%%%%%%%%%%%%%%%%%%%%%%%%%%%%%%%%%%%%%%%
\addcontentsline{toc}{section}{References}
\bibliographystyle{plain}
\bibliography{references-strongly-damped}
%%%%%%%%%%%%%%%%%%%%%%%%%%%%%%%%%%%%%%%%%%%%%
		
\end{document}